\documentclass[english]{smfart}

\usepackage{graphicx}
\usepackage{palatino}
\usepackage{parskip}
\usepackage{amsthm,amsfonts,amsmath,amssymb,amscd}
\usepackage{verbatim}
\usepackage{enumerate}
\usepackage{mathtools,bm,graphicx,tikz}

\renewcommand{\tilde}{\widetilde}

\newtheorem{theo}{Theorem}[section]
\newtheorem{lemm}[theo]{Lemma}
\newtheorem{coro}[theo]{Corollary}
\newtheorem{prop}[theo]{Proposition}

\theoremstyle{remark}
\newtheorem{rema}[theo]{Remark}

\theoremstyle{definition}
\newtheorem{defi}[theo]{Definition}
\newtheorem{exem}[theo]{Example}

\newcommand{\mb}[1]{\mathbf{#1}}
\newcommand{\CC}{\mb{C}}

\newcommand{\PP}{\mb{P}}
\newcommand{\QQ}{\mb{Q}}
\newcommand{\RR}{\mb{R}}
\newcommand{\ZZ}{\mb{Z}}

\newcommand{\OP}{\operatorname}

\newcommand{\cp}[1]{\CC\PP^{#1}}

\newcommand{\rt}{\bm{\mu}}

\newcommand{\mM}{\mathcal{M}}

\newcommand{\orb}[1]{\skew{3}\hat{#1}}

\mathchardef\ordinarycolon\mathcode`\:
\mathcode`\:=\string"8000
\begingroup \catcode`\:=\active
  \gdef:{\mathrel{\mathop\ordinarycolon}}
\endgroup

\makeatletter
\newcommand*{\da@rightarrow}{\mathchar"0\hexnumber@\symAMSa 4B }
\newcommand*{\da@leftarrow}{\mathchar"0\hexnumber@\symAMSa 4C }
\newcommand*{\xdashrightarrow}[2][]{%
  \mathrel{%
    \mathpalette{\da@xarrow{#1}{#2}{}\da@rightarrow{\,}{}}{}%
  }%
}
\newcommand{\xdashleftarrow}[2][]{%
  \mathrel{%
    \mathpalette{\da@xarrow{#1}{#2}\da@leftarrow{}{}{\,}}{}%
  }%
}
\newcommand*{\da@xarrow}[7]{%
  \sbox0{$\ifx#7\scriptstyle\scriptscriptstyle\else\scriptstyle\fi#5#1#6\m@th$}%
  \sbox2{$\ifx#7\scriptstyle\scriptscriptstyle\else\scriptstyle\fi#5#2#6\m@th$}%
  \sbox4{$#7\dabar@\m@th$}%
  \dimen@=\wd0 %
  \ifdim\wd2 >\dimen@
    \dimen@=\wd2 %
  \fi
  \count@=2 %
  \def\da@bars{\dabar@\dabar@}%
  \@whiledim\count@\wd4<\dimen@\do{%
    \advance\count@\@ne
    \expandafter\def\expandafter\da@bars\expandafter{%
      \da@bars
      \dabar@ 
    }%
  }%
  \mathrel{#3}%
  \mathrel{%
    \mathop{\da@bars}\limits
    \ifx\\#1\\%
    \else
      _{\copy0}%
    \fi
    \ifx\\#2\\%
    \else
      ^{\copy2}%
    \fi
  }%
  \mathrel{#4}%
}
\makeatother

\usetikzlibrary{decorations.pathreplacing}
\newcommand{\mC}{\mathcal{C}}
\newcommand{\mE}{\mathcal{E}}

\title[Bounds on Wahl singularities]{Bounds on Wahl singularities\\ from symplectic topology}

\author{Jonathan David Evans} \address{ Department of
  Mathematics\\ University College London\\ Gower
  Street\\ London\\ WC1E 6BT\\ United Kingdom}
\email{j.d.evans@ucl.ac.uk} \author{Ivan Smith} \address{Centre for
  Mathematical Sciences\\ University of Cambridge\\ Wilberforce
  Road\\ CB3 0WB\\ United Kingdom.}  \thanks{I.S. is partially
  supported by EPSRC Fellowship EP/N01815X/1.\\J.E. is supported by EPSRC
  Standard Grant EPSRC Grant EP/P02095X/1.} \email{is200@cam.ac.uk}

\begin{document}

\begin{abstract}
  Let $X$ be a minimal surface of general type with $p_g>0$ ($b^+>1$)
  and let $K^2$ be the square of its canonical class. Building on work
  of Khodorovskiy and Rana, we prove that if $X$ develops a Wahl
  singularity of length $\ell$ in a $\QQ$-Gorenstein degeneration,
  then $\ell\leq 4K^2+7$. This improves on the current best-known
  upper bound due to Lee ($\ell\leq 400(K^2)^4$). Our bound follows
  from a stronger theorem constraining symplectic embeddings of
  certain rational homology balls in surfaces of general type. In
  particular, we show that if the rational homology ball $B_{p,1}$
  embeds symplectically in a quintic surface, then $p\leq 12$, partially 
  answering the symplectic version of a question of Kronheimer. \newline   
\end{abstract}

\maketitle

\section{Introduction}

A complex surface is said to have general type if its canonical bundle
is big. The moduli space of surfaces of general type with fixed
characteristic numbers $K^2$ and $\chi$ admits a compactification,
constructed by Koll\'{a}r and Shepherd-Barron, whose boundary points
correspond to surfaces with {\em semi-log-canonical} (slc)
singularities, in much the way that the boundary points of
Deligne-Mumford space correspond to nodal curves. See \cite{Hacking}
for a survey of these moduli spaces.

The fact that this moduli space is compact was proved by Alexeev
\cite{Alexeev}, and is equivalent to the fact that there is a bound
(in terms of $K^2$ and $\chi$) on the index of the slc singularities
which appear at the boundary of the moduli space. A long-standing
question is to give explicit and effective bounds on the possible
indices in terms of the characteristic numbers. Existing upper bounds
are much weaker than one would expect from the known examples.

In this paper, we focus on cyclic quotient singularities of type
$\frac{1}{p^2}(pq-1,1)$ ({\em Wahl singularities}), and give a
boundedness result via an approach from symplectic topology.

To state the result, we recall that the minimal resolution of the Wahl
singularity $\frac{1}{p^2}(pq-1,1)$ has exceptional locus a chain of
spheres $C_1,\ldots,C_\ell$ with negative squares $C_i^2=-b_i$, where
\[\frac{p^2}{pq-1}=b_1-\frac{1}{b_2-\frac{1}{b_3-\cdots}}\]
is a continued fraction expansion. The number $\ell$ is called the
{\em length} of the singularity. The index of the singularity is $p$,
and is bounded above by $2^\ell$.

\begin{theo}\label{theo:alg_geom_version}
  Let $X$ be a minimal surface of general type with positive geometric
  genus $p_g>0$ which has a finite set of Du Val and Wahl
  singularities. Then the length $\ell$ of any of the Wahl
  singularities satisfies
  \[\ell\leq 4K_X^2+7.\]
\end{theo}

\begin{rema}
  The best bound for Wahl singularities we could find in the algebraic
  geometry literature\footnote{After the first version of this paper
    was prepared, in June 2017, we learned that a very similar
    (slightly stronger, and optimal) result had been proved by Rana
    and Urz\'{u}a (see \cite{RanaUrzua}). They
    use algebro-geometric rather than symplectic methods, so do not
    recover the purely symplectic Theorem \ref{theo:main} below. It is
    possible that their arguments can be used to strengthen our
    bounds.} is due to Lee \cite[Theorem 23]{Lee}, who showed that if
  $X$ is a stable surface of general type having a Wahl singularity of
  length $\ell$, and if the minimal model $S$ of the minimal
  resolution $\tilde{X}$ of $X$ has general type, then
  \[\ell\leq 400(K_X^2)^4.\]
  In another paper, Rana {\cite[Theorem 1.2]{Rana}} assumes further
  that the map $\tilde{X}\to S$ involves blowing down exactly $\ell-1$
  times and proves that, in this case, $\ell$ is either $2$ or
  $3$. She uses this to study the boundary of the
  Koll\'{a}r--Shepherd-Barron--Alexeev (KSBA) moduli space for
  surfaces with $K^2=\chi=5$ (e.g. quintic surfaces).
\end{rema}
  
\begin{rema}
  The hypothesis that $p_g>0$ (which is equivalent to $b^+(X)>1$) is
  there because our proof uses results from Seiberg-Witten theory
  and holomorphic curve theory which break down when $b^+(X)=1$ (or at
  least when the minimal model of $\tilde{X}$ is rational or
  ruled). The hypothesis entails that the minimal model $S$ is neither
  rational nor ruled, but allows, for example, that $S$ is an elliptic
  surface of positive genus (see Remark \ref{rema:lp} below for
  examples).
\end{rema}

Theorem \ref{theo:alg_geom_version} is really a theorem about
symplectic topology. The Wahl singularity $\frac{1}{p^2}(pq-1,1)$
admits a $\QQ$-Gorenstein smoothing whose smooth fibre is a symplectic
rational homology ball $B_{p,q}$ \cite{Wahl}. In \cite{Khodorovskiy},
Khodorovskiy conjectures that, for a surface $X$ of general type with
$b^+>1$, there is a bound on $p$ (depending only on $K_X^2$) such that
$B_{p,1}$ embeds symplectically in $X$. What we prove is the
following:

\begin{theo}[Generalised Khodorovskiy conjecture]\label{theo:main}
  Let $(X,\omega)$ be a symplectic 4-manifold with $K_X=[\omega]$ and
  $b^+(X)>1$. If there is a symplectic embedding $\iota\colon B_{p,q}\to X$,
  then
  \[\ell\leq 4K^2+7,\]
  where $\ell$ is the length of the continued fraction expansion of
  $p^2/(pq-1)$. In the special case $p=n$, $q=1$ (when $\ell=n-1$) we
  get the stronger inequality:
  \[\ell\leq 2K^2+1.\]
\end{theo}

\begin{rema}
  To see how Theorem \ref{theo:alg_geom_version} follows from this,
  note that if $X$ is a minimal surface of general type then its
  canonical model is a symplectic orbifold with at worst Du Val
  singularities where the canonical map has contracted a collection of
  $-2$-spheres. A Du Val singularity can be smoothed symplectically by
  excising a neighbourhood of the singularity and replacing it with a
  copy of the Milnor fibre of the singularity. The Milnor fibre of a
  Du Val singularity is a symplectic manifold which
  deformation-retracts onto a configuration of Lagrangian spheres (the
  vanishing cycles). The result is a symplectic 4-manifold
  $(X,\omega)$ with $K_X=[\omega]$; see {\cite[Proposition 3.1]{STY}}
  for more details. One can also smooth Wahl singularities
  symplectically, by replacing a neighbourhood of the singularity with
  a copy of the Milnor fibre $B_{p,q}$. Therefore if a surface of
  general type develops a Wahl singularity, one can find a symplectic
  4-manifold diffeomorphic to $X$ with $K_X=[\omega]$, admitting a
  symplectic embedding of $B_{p,q}$.

  For other papers which find obstructions to symplectic embeddings of
  rational homology balls, see \cite{EvansSmith,LM}. Theorem
  \ref{theo:main} could also be restated as a constraint on Lagrangian
  embeddings of pinwheels as in \cite{EvansSmith}.
\end{rema}

\begin{rema}\label{rema:varying_omega}
  Given a $\QQ$-Gorenstein degeneration $\mathcal{X} \to \Delta$ over
  the disc with central fibre $X_0$ of general type (so having ample
  canonical class) and having a Wahl singularity $1/p^2(1,pq-1)$, one
  constructs the embedding $B_{p,q} \hookrightarrow X_1$ into the
  general fibre by symplectic parallel transport.  If $X_0$ does not
  have ample canonical class, there is no global (orbifold) K\"ahler
  form in the class $K_{\mathcal{X}}$. In this case, one may need to
  perturb the (cohomology class of the) K\"ahler form in order to
  define parallel transport.

  In \cite{Hacking-flipping}, Hacking, Tevelev and Urz\'ua show that a
  surface of general type may admit degenerations in which infinitely
  many different Wahl singularities arise on the central fibre, where
  the central fibres are not ample (they are obtained by 3-fold flips
  from a $K$-ample degeneration). Translating back into the symplectic
  category, this implies that there are symplectic four-manifolds $X$
  with the property that $K_X=[\omega]$, and an open neighbourhood
  $K_X \in U \subset H^2(X;\RR)$, for which infinitely many rational
  balls $B_{p,q}$ admit symplectic embeddings into $(X,\omega_s)$ for
  some symplectic form with $[\omega_s] \in U$. (The parameter $s$
  will depend on the particular rational ball; in the setting of the
  previous paragraph, it is determined by the geometry of the flip.)
  Thus, not only is Theorem \ref{theo:main} an essentially symplectic
  rather than smooth phenomenon, but it is also sensitive to the
  cohomology class of the symplectic form\footnote{Added in proof: The
    recent paper \cite{EU} uses these ideas to find a single
    non-monotone symplectic form on a quintic surface which admits
    infinitely many symplectically embedded rational balls \(B_{p,q}\)
    with \(p\) unbounded.}.
\end{rema}

\begin{rema}
  In \cite{Khodorovskiy}, Khodorovskiy uses techniques from
  Seiberg-Witten theory and holomorphic curves to get strong
  restrictions on the way in which $-1$-spheres can intersect the
  curves $C_1,\ldots,C_\ell$ in the minimal resolution.  All of the
  techniques and ideas we use to prove Theorem \ref{theo:main} can be
  found in some form in Khodorovskiy's paper, and the structure of our
  case analysis is closely modelled on that in Rana's paper
  \cite{Rana}. The key trick which makes our case analysis easier than
  in Rana's seems to be the fact that we can perturb the almost
  complex structure on the minimal resolution so that the only
  somewhere-injective holomorphic curves present are the curves
  $C_1,\ldots,C_\ell$ and a finite set of embedded $-1$-spheres.
\end{rema}

\begin{rema}
  Theorem \ref{theo:main} is related to a question of Kronheimer
  {\cite[Problem 4.7]{Kronheimer}}, who asked: {\em For which $p$ does
    $B_{p,1}$ embed in the quintic surface?} If you ask for the
  embedding to be symplectic, then from Theorem \ref{theo:main}, we
  know that $p\leq 12$. Indeed, Theorem \ref{theo:main} applies to any
  surface of degree $d\geq 5$ in $\cp{3}$, which has $K^2=d(d-4)^2$
  and $p_g=\frac{d}{6}(d^2-6d+11)-1$. {\em Smooth} embeddings of
  $B_{p,1}$ into 4-manifolds are much more plentiful, see
  \cite{KhodorovskiySmooth,PPS}.
\end{rema}

\begin{rema}
The projective plane contains $B_{p,q}$ for infinitely many $p$, cf. \cite{EvansSmith}, so some hypotheses on the ambient manifold are certainly required to obtain a finiteness result.
\end{rema}

\begin{rema}
  It is unfortunate that we need to make the restriction $b^+(X)>1$
  (rather than just $K_X = [\omega]$) because huge numbers of examples
  have been constructed when $b^+=1$. See \cite{US} for an impressive
  list. The original recipe for constructing examples like these is
  due to Park \cite{Park} and Lee and Park \cite{LeeParkQG}: you blow
  up $\cp{2}$ a large number of times, find a configuration of curves
  $C_1,\ldots,C_\ell$ combinatorially equivalent to the exceptional
  locus of the minimal resolution of a Wahl singularity
  $\frac{1}{p^2}(pq-1,1)$, then contract and smooth to get a surface
  of general type with $b^+=1$ containing a symplectic embedding of
  $B_{p,q}$. It seems that it would require new ideas to prove a bound
  when $b^+=1$. The key input which fails when $b^+=1$ is Corollary
  \ref{coro:rat}. This is more than a technical issue: for a startling
  illustration of the geometric ramifications of this failure, see
  Remark \ref{rema:zhang}.
  \end{rema}

\begin{rema}\label{rema:lp}
  Happily, there are examples of Wahl singularities of type
  $\frac{1}{n-2}(n-3,1)$ in Horikawa surfaces $H(n)$ (again due to Lee
  and Park \cite{LeeParkHorikawa}) where the minimal model of the
  minimal resolution is an elliptic surface $E(n)$. In particular,
  these satisfy $b^+(H(n))=2n-1$ and $K_{H(n)}^2=4n-6$, so in this
  case the length $\ell=n-1$ is $\frac{1}{4}(K^2+2)$. In particular,
  we see that the best we can hope for is a bound which is linear in
  $K^2$. These examples are particularly interesting: the KSBA stable
  surface whose smoothing is \(H(n)\) has two Wahl singularities of
  type \(\frac{1}{n-2}(n-3,1)\) and they cannot be smoothed
  independently of one another, because the surface obtained by
  smoothing one and resolving the other violates the Noether
  inequality (see {\cite[Corollary 7.5]{FintushelSternRB}}). This
  shows one big advantage of the symplectic approach to bounding
  singularities: we can always rationally blow-down one of the
  singularities and bound the singularities one at a time.
\end{rema}

\begin{rema}
  The rational homology balls $B_{p,q}$ have played a prominent role
  in low-dimensional topology since the papers of Fintushel and Stern
  \cite{FintushelSternRB,FintushelSternKL4,FintushelSternDNN}, where
  they have been used to construct exotic 4-manifolds with small Betti
  numbers, starting with the paper \cite{Park}.
\end{rema}

\subsection{Structure of the paper}

In Section \ref{sct:ratblow}, we introduce the tools that we will use
in the rest of the paper. The key result is Corollary \ref{coro:rat},
which uses Seiberg-Witten theory to rule out the existence of certain
holomorphic rational curves in symplectic manifolds with $b^+>1$. This
is used in Sections \ref{sct:exccrvs}--\ref{sct:nestingcrvs} to prove
results about limits of sequences of holomorphically embedded
$-1$-spheres under Gromov compactness, and in Section
\ref{sct:iterblow} to find constraints on how other rational curves can
intersect these limits when $b^+>1$.

In Section \ref{sct:basics}, we remind the reader about the basic
properties of Wahl singularities and their minimal resolutions.

In Section \ref{sct:acs}, we define a class of almost complex
structures on the minimal resolution for which the only irregular
holomorphic curves are contained in a neighbourhood of the exceptional
locus $\mathcal{C}$ of the minimal resolution.

In Section \ref{sct:topology}, we recall the topological description
of the discrepancies of the minimal resolution, and use this to find a
constraint on the way holomorphic $-1$-spheres can intersect
$\mathcal{C}$ (Theorem \ref{theo:magic}).

In Sections \ref{sct:exclim}--\ref{sct:bounding}, we complete the case
analysis required to prove our inequality. The idea of the proof is
that the blow-down map $\tilde{X}\to S$ from the minimal resolution to
its minimal model must contain at least $\ell-K^2$ exceptional
curves. We show that (roughly) at least half of the exceptional curves
$E$ have $E\cdot\mathcal{C}\geq 2$ (``good curves'') and the rest of
the exceptional curves have $E\cdot\mathcal{C}=1$ (``bad
curves''). Once we prove that $\sum_{E}E\cdot\mathcal{C}\leq\ell+1$,
this implies the desired inequality. This part of the argument is
modelled heavily on Rana's paper {\cite[Lemmas 2.8 and 2.9]{Rana}}.

In Section \ref{sct:special}, we explain how in certain special cases,
one can prove that all curves are good, which gives an improved
inequality.

\subsection{Acknowledgements}

The authors would like to thank: Weiyi Zhang for extremely helpful
discussions about how $-1$-spheres can degenerate under Gromov
compactness -- see Remark \ref{rema:zhang} -- and how our paper might
generalise beyond the $p_g>0$ setting; Paul Hacking for useful
comments and for making us aware of the ongoing work of Rana and
Urz\'ua; Julie Rana and Giancarlo Urz\'ua for constructive discussions
once we learned about their work \cite{RanaUrzua}; and the anonymous referee for their comments.

\section{Rational curves and blowing down}\label{sct:ratblow}

We begin by reviewing some theorems about rational holomorphic curves
and exceptional curves in symplectic 4-manifolds with $b^+>1$. The
main result here is Corollary \ref{coro:rat}, which will be the main
technical tool later in the paper.

\subsection{Rational curves}\label{sct:ratcurves}

A  complex line bundle $L \to X$ on a smooth four-manifold $X$ is called \emph{characteristic} if $c_1(L)$ is an integer lift of $w_2(X)$.  Note that if $L$ is characteristic, then (using additive notation for line bundles as for divisors) $L+2L'$ is characteristic, for any $L'$.

\begin{theo}[Fintushel and Stern, {\cite[Theorem 1.3]{FintushelSternImmersedSpheres}}]\label{theo:FS}
  Let $X$ be a smooth 4-manifold with $b^+(X)>1$ and write $\sigma(X)$
  for the signature of $X$ and $e(X)$ for the Euler characteristic of
  $X$. Let $L$ be a characteristic line bundle on $X$ with
  nonvanishing Seiberg-Witten invariant $SW_X(L)\neq 0$; we will also
  write $L$ for the first Chern class of $L$. Suppose that
  the virtual dimension
  \[\dim\mM_X(L)=\frac{1}{4}(L^2-(3\sigma(X)+2e(X)))\]
  of the moduli space of solutions to the Seiberg-Witten equations can
  be written as $\sum\ell_i(\ell_i+1)$ for some collection of
  nonnegative integers $\ell_1,\ldots,\ell_r$. If a nonzero homology
  class $A$ can be represented by an immersed sphere with $p$ positive
  double points then either:
  \[2p-2\geq \begin{cases}
    A^2+|L\cdot A|+4\sum_{i=1}^r\ell_i,&\quad p\geq r\\
    A^2+|L\cdot A|+4\sum_{i=1}^p\ell_i+2\sum_{i=p+1}^r\ell_i,&\quad p>r
  \end{cases}\]
  or
  \[
  SW_X(L)=\begin{cases}
  SW_X(L+2A),& A\cdot L\geq 0\\
  SW_X(L-2A),& A\cdot L\leq 0.
  \end{cases}
  \]
\end{theo}

We will use Theorem \ref{theo:FS} to prove the following result, for which
we could not find a proof in the literature:

\begin{coro}\label{coro:rat}
  Let $(X,J)$ be an almost complex 4-manifold with $b^+(X)>1$ and
  suppose that the canonical class $K$ associated to $J$ has
  $SW(K)\neq 0$ (for example, this holds if $J$ is homotopic to an
  $\omega$-tame almost complex structure for some symplectic form
  $\omega$ on $X$). Suppose that $u\colon S^2\to X$ is a
  somewhere-injective $J$-holomorphic curve representing a homology
  class $A$. If $K\cdot A\leq -1$ then $A^2=-1$, $K\cdot A=-1$. In
  this case, by the adjunction formula, $u$ is an embedding.
\end{coro}

\begin{proof}
  
  The adjunction formula {\cite[Theorem 1.3]{McDuffLocalBehaviour}},
    {\cite[Theorem 2.2.1]{McDuffSingularities}} tells us that
  \[A^2+2+K\cdot A\geq 0\]
  with equality if and only if $u$ is an embedding. If
  $K\cdot A\leq -1$ then this means $A^2\geq -1$ with equality if
  and only if $K\cdot A=-1$ and $u$ is an embedding. Henceforth
  we will assume that $A^2\geq 0$, $K\cdot A\leq -1$ and derive
  a contradiction.

  We will prove by induction on $m$ that:
  \begin{equation*}\label{eq:SWinduction}
    SW_X(K-2mA)\neq 0\mbox{ for all nonnegative integers }m.
  \end{equation*}
  This will give a contradiction, since the set of cohomology classes
  with nonvanishing Seiberg-Witten invariant is finite {\cite[Section
      3]{Witten}}. The base case $m=0$ holds by assumption. The fact
  that this holds whenever $J$ is homotopic to an $\omega$-tame almost
  complex structure for some symplectic form $\omega$ is a result of
  Taubes {\cite[Main Theorem]{TaubesSWsymp}}. We now assume that
  $SW_X(K-2mA)\neq 0$.
  
  By {\cite[Proposition 1.2]{McDuffLocalBehaviour}} or {\cite[Theorem
      4.1.1]{McDuffSingularities}}, there is a homotopic $J'$ and
  a $\mC^\infty$-small perturbation $u'$ of $u$ such that $u'$
  is a $J'$-holomorphic immersion with positive transverse double
  points. Therefore, without loss of generality, we may assume that
  $u$ is an immersion with positive transverse double points. Let $p$
  be the number of positive transverse double points of $u$. If we
  smooth the double points we find an embedded symplectic surface
  $\Sigma$ with genus $p$ in the homology class $A$; by the standard
  adjunction formula, we have $A^2+2-2p=-K\cdot A$, or
  \begin{equation}\label{eq:adj}
    2p-2=A^2+K\cdot A.
  \end{equation}

  We will apply Theorem \ref{theo:FS} to $u$ with $L_m=K-2mA$. We have
  \[L_m\cdot A=K\cdot A-2mA^2=-c_1(A)-2mA^2\leq 0,\]
  as we are assuming $K\cdot A\leq -1$ and $A^2\geq 0$. Theorem
  \ref{theo:FS} then tells us that either $SW_X(L_{m+1})\neq 0$ (which
  would complete the induction step) or
  \[2p-2\geq A^2+|L_m\cdot A|+\mbox{nonnegative terms}\geq A^2.\]
  However, from Equation \eqref{eq:adj}, we know that
  $2p-2=A^2+K\cdot A<A^2$. This completes the induction step.

\end{proof}

\begin{rema}
  We observe the following result of Zhang \cite{Zhang2}:

  \begin{lemm}[{\cite[Lemma 2.1]{Zhang2}}]\label{lemm:rat}
    Let $(X,\omega)$ be a symplectic 4-manifold whose minimal model is
    neither rational nor ruled and let $J$ be an $\omega$-tame almost
    complex structure on $X$. If $C\subset X$ is a somewhere-injective
    holomorphic curve whose domain is a compact, connected Riemann
    surface and $K\cdot C<0$, then $C$ is an embedded sphere with
    self-intersection $-1$.
  \end{lemm}

  This also relies on Seiberg-Witten theory in an essential way, via
  result of Taubes \cite{TaubesSWGW} and Li-Liu \cite{LiLiuUniqueness}
  which prove existence of a $J$-holomorphic representative of twice
  the canonical class. It seems likely that one could use this lemma
  in our proof to replace the $b^+>1$ assumption by the weaker
  assumption that the minimal model of the minimal resolution is
  rational or ruled, however, there is one small technical hurdle (in
  the proof of Proposition \ref{prop:excfirst}) which currently
  requires us to work with non-tame almost complex structures.

  If we assume that $X$ is a projective surface then Lemma
  \ref{lemm:rat} is a very classical result \cite{BHPV}.
  McDuff {\cite[Theorem 1.4]{McDuffImmersedSpheres}} was the first to
  prove a symplectic version of this, under the stronger assumption
  that $K\cdot C<-1$.
\end{rema}

A first (well-known) consequence of Corollary \ref{coro:rat} is:

\begin{lemm}\label{lemm:distinctEclasses}
  Let $(X,\omega)$ be a symplectic 4-manifold with $b^+(X)>1$. Suppose
  that $E_1$ and $E_2$ are homology classes in $H_2(X)$ which can be
  represented by embedded symplectic spheres, and suppose that
  $E_1^2=E_2^2=-1$. Then either $E_1=E_2$ or $E_1\cdot E_2=0$.
  Moreover, there are only finitely many Hamiltonian isotopy classes
  of symplectically embedded $-1$-spheres.
\end{lemm}
\begin{proof}
  The Gromov invariants of the class $E_1$ and of the class $E_2$ are
  both $1$, so for a generic almost complex structure $J$, these
  classes can be represented by $J$-holomorphic spheres $S_1$ and
  $S_2$ respectively. These spheres must either be identical or else
  disjoint: otherwise, we could blow $S_1$ down, and the image $S'_2$
  of $S_2$ under the blow-down would be a rational curve with
  $K\cdot S'_2=K\cdot S_2-S_1\cdot S_2<-1$, contradicting Corollary
  \ref{coro:rat}. Therefore, $E_1=E_2$ or $E_1\cdot E_2=0$.

  The finiteness statement is clear on the level of homology: there
  can be at most $b^-(X)$ pairwise orthogonal $-1$-classes in the
  homology of $X$. Moreover, if there are two homologous
  symplectically embedded $-1$-spheres $S_1$ and $S_2$, then we can
  pick a $J_1$ making $S_1$ holomorphic and a $J_2$ making $S_2$
  holomorphic. A generic path $J_t$ connecting $J_1$ and $J_2$ in the
  space of compatible almost complex structures will avoid the
  codimension 2 locus where the spheres in this class bubble, and the
  unique $J_t$-holomorphic sphere $S_t$ in this homology class traces
  out a Hamiltonian isotopy connecting $S_1$ and $S_2$. Therefore
  there are only finitely many Hamiltonian isotopy classes of
  symplectically embedded $-1$-spheres.
\end{proof}

\begin{rema}
  All of the facts proved here will fail if the surface is rational
  (except finiteness of the number of $-1$-classes if the surface is
  Del Pezzo).
\end{rema}

\subsection{Exceptional curves of the first kind}\label{sct:exccrvs}

In this subsection, we review some basic material about birational
maps of complex surfaces, and extend the theory to handle symplectic
4-manifolds.

\begin{defi}
  An {\em exceptional curve of the first kind} in a complex surface
  $X$ is a (possibly reducible, nonreduced) divisor $E$ for which
  there is a holomorphic birational map $\pi\colon X\to Y$ to a smooth
  complex surface $Y$ and a point $p\in Y$ such that
  $\pi^{-1}(p)=E$.
\end{defi}

An embedded sphere $E$ with $E^2=-1$ (or, equivalently,
$K\cdot E=-1$) is an exceptional curve of the first kind.
Any irreducible
exceptional curve of the first kind has this form. Any exceptional
curve of the first kind with $m$ irreducible components can be
obtained by taking an exceptional curve of the first kind with $m-1$
irreducible components, blowing up a point on the curve, and taking
the total transform.

Here are some examples. In each, we will represent the exceptional
curve by drawing a graph with a vertex for each irreducible
component; we label the vertex associated to the component $C$ with
the integer $K\cdot C$. We also write the name of the component and
its multiplicity above the vertex.

\begin{exem}
  We start with a $-1$-sphere $E$. We will denote the proper transform
  of a curve $C$ under a blow-up by $\tilde{C}$.
  \begin{enumerate}
  \item[(1)] Blow-up a point on $E$ and take total transform to get a
    curve:    
    \tikz[baseline=0]{
      \draw[fill=black,radius=2pt]
      (0,0) circle node [above=2pt] {$E_1$} node [below=2pt] {$-1$}
      -- (1,0) circle node [above=2pt] {$E_2$} node [below=2pt] {$0$};
    }
    
  \item[(2)] Blow-up a point $p$ on the previous example. Let $F_1$ be
    the new $-1$-sphere, let $F_1=\tilde{E}_2$ and
    $F_3=\tilde{E}_2$. There are three possibilities:
    
    \begin{tabular}[htb]{p{1cm}p{3cm}p{5cm}}
      (2.1) & $p\in E_1$, $p\not\in E_2$ &
      \tikz[baseline=0]{
        \draw[fill=black,radius=2pt] (0,0)
        circle node [above=2pt] {$F_1$} node [below=2pt] {$-1$}
        --
        (1,0) circle node [above=2pt] {$F_2$} node [below=2pt] {$0$}
        --
        (2,0) circle node [above=2pt] {$F_3$} node [below=2pt] {$0$};
      }\\[15pt]
      (2.2) & $p\not\in E_1$, $p\in E_2$ &
      \tikz[baseline=0]{
        \draw[fill=black,radius=2pt] (0,0)
        circle node [above=2pt] {$F_2$} node [below=2pt] {$-1$}
        --
        (1,0) circle node [above=2pt] {$F_3$} node [below=2pt] {$1$}
        --
        (2,0) circle node [above=2pt] {$F_1$} node [below=2pt] {$-1$};
      }\\[15pt]
      (2.3) & $p\in E_1\cap E_2$ &
      \tikz[baseline=0]{
        \draw[fill=black,radius=2pt] (0,0)
        circle node [above=2pt] {$F_2$} node [below=2pt] {$0$}
        --
        (1,0) circle node [above=2pt] {$2F_1$} node [below=2pt] {$-1$}
        --
        (2,0) circle node [above=2pt] {$F_3$} node [below=2pt] {$1$};
      }
    \end{tabular}
  \end{enumerate}
\end{exem}

The following properties of exceptional curves of the first kind are
easy to prove by induction on the number of components of the
curve. For a full treatment of exceptional curves of the first kind,
see \cite{BarberZariski}.

\begin{theo}\label{theo:zariski}
  \begin{enumerate}
    \item[(1)] Every irreducible component of an exceptional curve of
      the first kind $E$ is an embedded sphere with negative
      self-intersection number.
    \item[(2)] Any two irreducible components intersect at most once,
      transversely.
    \item[(3)] If $G$ is the {\em dual graph} whose vertices
      correspond to irreducible components of $E$ and whose edges
      correspond to intersections between components, then $G$ is a
      connected tree.
    \item[(4)] We can factor the blow-down map $\pi\colon X\to Y$ as a
      sequence
      \[X=X_1\stackrel{\pi_1}{\to}
      X_2\stackrel{\pi_2}{\to}X_3\to\cdots\stackrel{\pi_n}{\to}
      X_{n+1}=Y\] where each $\pi_i$ blows down one $-1$-sphere. Let
      $\Pi_i=\pi_i\circ\pi_{i-1}\circ\cdots\circ\pi_1$.  We call $A_i$
      the component of $E$ which is contracted by $\Pi_i$ but not
      $\Pi_{i-1}$. Then, for every $i$, there is at most one $j>i$ for
      which $A_i\cdot A_j\neq 0$.
    \item[(5)] If we write $E=\sum_{i=1}^nm_iA_i$ then
      \[m_i=\sum_{j<i}m_jA_j\cdot A_i\]
    \item[(6)] We have $E\cdot A_i=0$ for $i<n$ and $E\cdot A_n=-1$.
      In particular, $E\cdot\sum_{i=1}^nA_i=-1$.
    \item[(7)] There is at least one $-1$-sphere amongst the
      irreducible components, and every $-1$-sphere component can
      intersect at most $2$ other components.
  \end{enumerate}
\end{theo}

\subsection{Exceptional curves in symplectic manifolds}\label{sct:sympexccrvs}

In what follows, we use an analogue of exceptional curves of the first
kind in symplectic topology.

\begin{defi}
  Let $(X,\omega)$ be a symplectic 4-manifold. Let $\Sigma$ be a nodal
  Riemann surface of genus zero and $u\colon\Sigma\to X$ be a
  continuous map from a nodal Riemann surface. We say that $u$ is an
  exceptional curve of the first kind if there exists:
  \begin{itemize}
  \item a neighbourhood $M$ of $u(\Sigma)$ and an open neighbourhood
    $N$ of $0\in\CC^2$;
  \item an integrable complex structure $J$ on $M$ homotopic to an
    $\omega$-tame one;
  \item a holomorphic birational map $\pi\colon M\to N$;
  \end{itemize}
  such that $u$ is a $J$-holomorphic stable map and
  $\pi^{-1}(0)=u(\Sigma)$.
\end{defi}

\begin{prop}\label{prop:excfirst}
  Suppose that $(X,\omega)$ is a symplectic 4-manifold with $b^+(X)>1$
  and $J$ is an $\omega$-tame almost complex structure. If $u$ is a
  $J$-holomorphic stable map representing a homology class $E$ with
  $K\cdot E=-1$, and if $E$ can be represented by an embedded
  symplectic 2-sphere, then $u$ is an exceptional curve of the first
  kind.
\end{prop}
\begin{proof}
  If $J$ is an arbitrary almost complex structure on $X$ and $u$ is
  $J$-holomorphic, then, by {\cite[Theorem 3]{Sikorav}}, there exists
  another (homotopic, but not obviously tame\footnote{If one could
    prove that Sikorav's result yields a tame complex structure then
    we could appeal to Lemma \ref{lemm:rat} instead of Corollary
    \ref{coro:rat} and deduce all of our results in the more general
    setting that the minimal model is not rational or ruled. Added in
    proof: This improvement was recently made by Chen and Zhang
    {\cite[Appendix A]{ChenZhang}}.}) almost complex structure $J'$
  for which $u$ is $J'$-holomorphic and $J'$ is integrable on a
  neighbourhood of the image of $u$. We may therefore assume without
  loss of generality that $J$ is integrable on a neighbourhood of
  $u(\Sigma)$.

  We will prove the proposition by induction on the number of
  irreducible components of the image of $u$. The induction step
  involves blowing down existing $-1$-spheres. To blow-down, we use
  integrability of $J$ near the exceptional curve. We observe that
  Corollary \ref{coro:rat} holds for $J$-curves in $X$ because $J$ is
  homotopic to a tame almost complex structure; it continues to hold
  for blow-downs because the blow-down formula for Seiberg-Witten
  invariants implies that the canonical class is still a Seiberg-Witten
  basic class.

  \begin{lemm}[cf {\cite[Corollary 2.10]{Zhang2}}]
    \label{lemm:minusonesphere}
    If the image of $u$ has $n\geq 1$ components then one of them is
    an embedded $-1$-sphere $e$.
  \end{lemm}
  \begin{proof}
    Let $A_1,\ldots,A_n$ be simple holomorphic curves underlying the
    irreducible components of the image of $u$ and $m_1,\ldots,m_n$ be
    the covering multiplicities, so that $E=\sum_{i=1}^nm_iA_i$. We
    have
    \[-1=K\cdot E=\sum_{i=1}^n m_iK\cdot A_i,\]
    so at least one of the numbers $K\cdot A_i$ is negative. By
    Lemma \ref{coro:rat}, $A_i^2=-1$ and this component is an embedded
    $-1$-sphere.
  \end{proof}

  We write $\pi\colon X\to X'$ for the holomorphic map which blows
  down the curve $e$ produced by Lemma \ref{lemm:minusonesphere}. After
  possibly adding marked points to the domain, the composition
  $\pi\circ u$ is a stable map representing a homology class
  $\pi_*(E)$.

  If $E$ and $e$ are distinct classes then, by Lemma
  \ref{lemm:distinctEclasses}, they satisfy $E\cdot e=0$.
  
  \begin{lemm}
    Either $\pi\circ u$ is constant or $K_{X'}\cdot\pi_*(E)=-1$.
  \end{lemm}
  \begin{proof}
    Recall that $E=\sum_{i=1}^nm_iA_i$ and that one of the $A_i$, say
    $A_1$, is a class $e$ with $e^2=-1$. By Lemma
    \ref{lemm:distinctEclasses}, either $E=e$ or $E\cdot e=0$. If $E=e$
    then $\pi\circ u$ is constant, so we may assume $E\cdot e=0$.
    Then:
    \[0=E\cdot e=-m_1+\sum_{i=2}^nm_iA_i\cdot e,\]
    so $m_1=\sum_{i=2}^nm_iA_i\cdot e$. Moreover,
    \[K_X\cdot E=m_1K_X\cdot e+\sum_{i=2}^nK_X\cdot A_i,\]
    so
    \begin{equation}\label{eq:blowdown}
      -1=\sum_{n=2}^nm_i(A_i\cdot e+K_X\cdot A_i).
    \end{equation}
    The expression $K_X\cdot A_i+e\cdot A_i$ is equal to
    $K_{X'}\cdot\pi_*(A_i)$. Therefore Equation \eqref{eq:blowdown}
    tells us that $K_{X'}\cdot\pi_*(E)=-1$.
  \end{proof}

  By induction, this implies that $u$ can be blown down to a point;
  therefore $u$ is an exceptional curve of the first kind.
\end{proof}

\begin{rema}\label{rema:zhang}
  Proposition \ref{prop:excfirst} fails if the minimal model is
  rational. We are grateful to Weiyi Zhang for pointing out the
  following wonderful example. Let $C$ be a rational plane quartic
  curve with three nodes. Blow up the three nodes, along with five
  other points on the curve. The proper transform of $C$ is an
  embedded symplectic $-1$-sphere in the homology class
  \[E:=4H-2E_1-2E_2-2E_3-E_4-E_5-E_6-E_7-E_8,\]
  which satisfies $K\cdot E=-1$ and $E^2=-1$. However, for a
  nongeneric complex structure, this can be represented by a stable
  map which is not an exceptional curve of the first kind: the
  arithmetic genus of the image is one. To see this, take a line $C_1$
  and a conic $C_2$ which intersect at two points, and blow-up three
  points on the line and five on the conic. The proper transform of
  $2C_1+C_2$ lives in the class $E$; it can be represented by a stable
  map whose domain has three components $a,b,c$ in a chain which map
  respectively to $C_1,C_2,C_1$. This is not an exceptional curve of
  the first kind; for example its components intersect one another
  twice. For many counter-intuitive examples of $J$-holomorphic
  subvarieties of rational surfaces, as well as some constraints for
  nef classes or certain ruled surfaces, we refer the interested
  reader to papers by Li-Zhang and Zhang \cite{LiZhang,Zhang2,Zhang1}.
\end{rema}

\subsection{Nesting of exceptional curves}\label{sct:nestingcrvs}

\begin{defi}
  If two exceptional curves of the first kind, $E_1$ and $E_2$, have
  the property that all of the components of $E_1$ are also components
  of $E_2$, then we say $E_1\subset E_2$, or that the classes are {\em
    nested}.
\end{defi}

\begin{prop}\label{prop:nesting}
  If $(X,\omega)$ is a symplectic manifold with $b^+(X)>1$ and two
  exceptional curves of the first kind $E_1,E_2\subset X$ share a
  component, then they are nested.
\end{prop}
\begin{proof}
  Given a pair of exceptional curves of the first kind, $E_1,E_2$,
  having $N_{E_1}$, respectively $N_{E_2}$, components, define
  \[M_{E_1E_2}=\max(N_{E_1},N_{E_2}).\]
  If $M_{E_1E_2}=1$ then both $E_1$ and $E_2$ are irreducible, so, if
  they share a component, then they are certainly nested (indeed they
  are geometrically indistinct). Let us assume, as an induction
  hypothesis, that any pair of exceptional curves of the first kind
  $E_1$ and $E_2$ in any symplectic manifold $X$ with $b^+(X)>1$, which
  share a component and satisfy $M_{E_1E_2}<m$, are nested.

  Let $E_1$ and $E_2$ be a pair of exceptional curves of the first
  kind with $M_{E_1E_2}=m$ and which share a component. Suppose for
  a contradiction that they are not nested. Then there exists a
  component $A\subset E_1$ which intersects $E_2$ but is not contained
  in it. This component cannot be a $-1$-sphere, as it would have
  positive intersection number with $E_2$, which would contradict
  Lemma \ref{lemm:distinctEclasses}. Therefore we can blow down all the
  $-1$-spheres in $E_1$ and $E_2$ to obtain a non-nested configuration
  $E'_1$, $E'_2$ with $M_{E'_1E'_2}<m$. This contradicts the induction
  hypothesis, so we deduce that $E_1$ and $E_2$ are nested.
\end{proof}

\subsection{Iterated blow-down of rational curves}\label{sct:iterblow}

The following lemma will be useful in streamlining our arguments
later.  (Our convention in the figures is that a black box denotes a component of square $-1$.)

\begin{lemm}\label{lemm:iteratedblowdowns}
  Let $(X,\omega)$ be a symplectic manifold with $b^+(X)>1$. Suppose
  that $E$ is an exceptional curve of the first kind which is just a
  chain of spheres $F_1,\ldots,F_n$ with a single $-1$-sphere $F_i$,
  $i\in\{2,\ldots,n-1\}$. Suppose that $S$ is a rational curve which
  intersects $F_i$ once transversely and is disjoint from the other
  components of $E$. Then, after blowing down $E$, the image $T$ of
  $S$ is a rational curve with $K\cdot T\leq K\cdot
  S-2(n-1)$. In particular, by Corollary \ref{coro:rat}, we have
  \[n\leq\frac{1}{2}(K\cdot S+3).\]

  
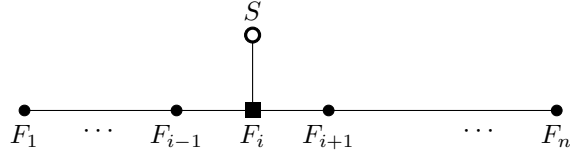
\begin{center}
  \tikz[baseline=0]{
    \draw[fill=black,radius=2pt]
    (1,0) circle node [below=2pt] {$F_1$}
    --
    (2,0) node [below=2pt] {$\cdots$}
    --
    (3,0) circle node [below=2pt] {$F_{i-1}$}
    --
    (4,0) node [below=2pt] {$F_i$}
    --
    (5,0) circle node [below=2pt] {$F_{i+1}$}
    --
    (7,0) node [below=2pt] {$\cdots$}
    --
    (8,0) circle node [below=2pt] {$F_n$};
    \fill (3.9,-0.1) rectangle (4.1,0.1);
    \draw[fill=black,radius=3pt]
    (4,1) circle node [above=2pt] {$S$} -- (4,0);
    \fill[fill=white,radius=2pt] (4,1) circle;
  }
  \end{center}
\end{lemm}

\begin{proof}
  To simplify notation, we will begin by considering a special
  case. We take $i=2$, and suppose that $F_3,\ldots,F_n$ are all
  $-2$-spheres. As a diagram, labelling each vertex by $K\cdot F_i$,
  that means:

  
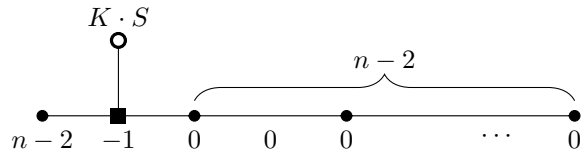
\begin{center}
    \tikz[baseline=0]{
      \draw[fill=black,radius=2pt]
      (1,0) circle node [below=2pt] {$n-2$}
      --
      (2,0) node [below=2pt] {$-1$}
      --
      (3,0) circle node [below=2pt] {$0$}
      --
      (4,0) node [below=2pt] {$0$}
      --
      (5,0) circle node [below=2pt] {$0$}
      --
      (7,0) node [below=2pt] {$\cdots$}
      --
      (8,0) circle node [below=2pt] {$0$};
      \fill (1.9,-0.1) rectangle (2.1,0.1);
      \draw[fill=black,radius=3pt]
      (2,1) circle node [above=2pt] {$K\cdot S$} -- (2,0);
      \fill[fill=white,radius=2pt] (2,1) circle;
      \draw [decorate,decoration={brace,amplitude=10pt}]
        (3,0.2) -- (8,0.2) node [midway,yshift=15pt] {$n-2$};
    }
  \end{center}

  When we blow down $F_2$, we get the following diagram:

  
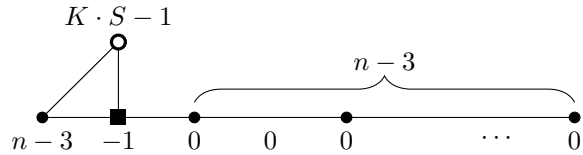
\begin{center}
    \tikz[baseline=0]{
      \draw[fill=black,radius=2pt]
      (1,0) circle node [below=2pt] {$n-3$}
      --
      (2,0) node [below=2pt] {$-1$}
      --
      (3,0) circle node [below=2pt] {$0$}
      --
      (4,0) node [below=2pt] {$0$}
      --
      (5,0) circle node [below=2pt] {$0$}
      --
      (7,0) node [below=2pt] {$\cdots$}
      --
      (8,0) circle node [below=2pt] {$0$};
      \fill (1.9,-0.1) rectangle (2.1,0.1);
      \draw[fill=black,radius=3pt]
      (1,0)
      --
      (2,1) circle node [above=2pt] {$K\cdot S-1$}
      --
      (2,0);
      \fill[fill=white,radius=2pt] (2,1) circle;
      \draw [decorate,decoration={brace,amplitude=10pt}]
      (3,0.2) -- (8,0.2) node [midway,yshift=15pt] {$n-3$};
    }
  \end{center}

  and when we blow down the next component, we get

  
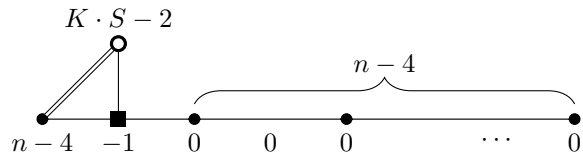
\begin{center}
    \tikz[baseline=0]{
      \draw[fill=black,radius=2pt]
      (1,0) circle node [below=2pt] {$n-4$}
      --
      (2,0) node [below=2pt] {$-1$}
      --
      (3,0) circle node [below=2pt] {$0$}
      --
      (4,0) node [below=2pt] {$0$}
      --
      (5,0) circle node [below=2pt] {$0$}
      --
      (7,0) node [below=2pt] {$\cdots$}
      --
      (8,0) circle node [below=2pt] {$0$};
      \fill (1.9,-0.1) rectangle (2.1,0.1);
      \draw[double,double distance=1pt]
      (1,0) -- (2,1);
      \draw[fill=black,radius=3pt]
      (2,1) circle node [above=2pt] {$K\cdot S-2$}
      --
      (2,0);
      \fill[fill=black,radius=2pt] (1,0) circle;
      \fill[fill=white,radius=2pt] (2,1) circle;
      \draw [decorate,decoration={brace,amplitude=10pt}]
      (3,0.2) -- (8,0.2) node [midway,yshift=15pt] {$n-4$};
    }
  \end{center}

  where the double line indicates an intersection multiplicity of
  $2$. After blowing down all the components $F_2,\ldots,F_n$, we are
  left with a single $-1$-sphere which intersects the blow-down $S'$
  of $S$ at this stage with multiplicity $n-1$. Moreover,
  $K\cdot S'=K\cdot S-(n-1)$. When we blow down the final
  $-1$-sphere, since $S'$ intersects it with multiplicity $n-1$, the
  resulting blowdown $T$ of $S'$ has
  \[K\cdot T=K\cdot S-2(n-1).\]
  Since $K\cdot T\geq -1$ by Corollary \ref{coro:rat}, we deduce
  \begin{equation}\label{eq:moss}
    n\leq\frac{1}{2}(K\cdot S+3).
  \end{equation}
  
  In this special case, at every stage until the last, we were blowing
  down the component of $E$ to the right of the previous one. For any
  other exceptional curve $E$, we would occasionally need to shift
  direction and blow down the curve to the left. This only serves to
  make Inequality \eqref{eq:moss} stronger. To explain why, observe
  that:
  \begin{itemize}
  \item Let $S^{(k)}$ denote the blow-down of $S$ at stage $k$
    (i.e. when $k$ components of $E$ have been blown down). At each
    stage except the first and last, $S^{(k)}$ intersects precisely
    two of the remaining spheres in the chain (the two which were
    adjacent to the curve which was blown down at the previous
    stage). Let us call these $F^{(k)}_L$ and $F^{(k)}_R$ (for
    ``left'' and ``right'').
  \item One of these two spheres $F^{(k)}_L$ or $F^{(k)}_R$ is
    necessarily the next curve to be blown down, say without loss of
    generality $F^{(k)}_L$. After blowing down $F^{(k)}_L$, we have
    \begin{align*}
      K\cdot S^{(k+1)}&=K\cdot S^{(k)}-F^{(k)}_L\cdot S^{(k)}\\
      F^{(k+1)}_L\cdot S^{(k+1)}&=F^{(k)}_L\cdot S^{(k)}\mbox{ (if }F^{(k+1)}_L\mbox{ exists)}\\
      F^{(k+1)}_R\cdot S^{(k+1)}&=(F^{(k)}_L+F^{(k)}_R)\cdot S^{(k)}
    \end{align*}
    so the intersection number of $S^{(k)}$ with $F^{(k)}_L$ and
    $F^{(k)}_R$ increases by one each time, starting at one when
    $k=1$. We see that the maximum of the intersection number between
    $S^{(k)}$ and the remaining spheres in the chain is always at
    least $k$.
  \item Each blow-down therefore reduces $K\cdot S^{(k)}$ by at least
    one, and the final blow-down reduces it by at least $n-1$ because
    $S$ hits the final sphere with multiplicity at least $n-1$. Again,
    we deduce Inequality \eqref{eq:moss}.
  \end{itemize}
\end{proof}

\section{Basic notions}\label{sct:basics}

\subsection{Definitions}

In this section, we fix the notation for the rest of the paper.

{\bf Wahl singularities:} Given coprime positive integers $p,q$, let
$\rt_{p^2}$ denote the group of $p^2$th roots of unity and consider
the action $\Gamma_{p,q}$ of $\rt_{p^2}$ on $\CC^2$ where a root $\mu$
acts as
\[\mu\cdot(x,y)=(\mu^{pq-1} x,\mu y).\]
The cyclic quotient singularity $\CC^2/\Gamma_{p,q}$ is called a {\em
  Wahl singularity} and is conventionally written
$\tfrac{1}{p^2}(pq-1,1)$. It has the property that its Milnor fibre
$B_{p,q}$ is a rational homology ball; here the Milnor fibre is a
compact Stein domain obtained by taking a compact neighbourhood of the
singularity and smoothing the singular point. Let $\Sigma_{p,q}$ be
the boundary of $B_{p,q}$; this is a contact hypersurface
contactomorphic to $S^3/\Gamma_{p,q}\subset\CC^2/\Gamma_{p,q}$.

{\bf $X$, $\orb{X}$:} Let $(X,\omega)$ be a symplectic manifold and
suppose that there is a symplectic embedding $\iota\colon
B_{p,q}\hookrightarrow X$. Let $\orb{X}$ denote the symplectic
orbifold obtained by excising $\iota(B_{p,q})$ from $X$ and replacing
it with a neighbourhood of the singular point in
$\CC^2/\Gamma_{p,q}$. We will say that $\orb{X}$ is {\em obtained from
  $X$ by collapsing the image of $\iota$}.

{\bf The minimal resolution $\tilde{X}$:} The singularity
$\tfrac{1}{p^2}(pq-1,1)$ has a minimal resolution: if
\[\frac{p^2}{pq-1}=b_1-\frac{1}{b_2-\frac{1}{b_3-\cdots}}\]
is the continued fraction expansion of $p^2/(pq-1)$ then the
exceptional divisor is a chain of spheres $C_1,\ldots,C_{\ell}$ where
\[
C_i\cdot C_j=\begin{cases} 1 &\mbox{ if }|i-j|=1\\
-b_i &\mbox{ if }i=j\\
0 &\mbox{ otherwise.}
\end{cases}
\]
We call the number $\ell$ of spheres in this chain the {\em length} of
the singularity. We will write $\mC$ for the set
$\{C_1,\ldots,C_\ell\}$. The process of collapsing the image of
$\iota$ and then taking the minimal resolution makes sense purely
symplectically
\cite{FintushelSternRB,Symington1,Symington2,KhodorovskiySRB} and is
called {\em generalised rational blow-up}. We denote the generalised
rational blow-up of $X$ along $\iota$ by $\tilde{X}$, and will usually
just refer to this as the minimal resolution of $\orb{X}$. Note that
it is only determined up to symplectic deformation equivalence: for
instance, one has to choose the symplectic areas of the curves $C_j$.

{\bf The minimal model $S$:} While $\tilde{X}\to\orb{X}$ is the minimal
resolution (in the sense that its exceptional divisor contains no
components which can be contracted smoothly), it may not be minimal
(in the sense that it may contain exceptional curves of the first kind
which are not contained in the exceptional divisor). Let $S$ be the
symplectic minimal model of $\tilde{X}$. By {\cite[Theorem
    1.1(i)]{McDuffRationalRuled}}, $S$ is obtained from $\tilde{X}$ by
blowing down a maximal collection of disjoint embedded symplectic
$-1$-spheres. Again, it is only determined up to symplectic
deformation equivalence.

\subsection{Combinatorics of Wahl singularities}

\begin{defi}
  We call a string $[b_1,\ldots,b_\ell]$ a {\em T-string} if it
  arises as $b_j=-C_j^2$ for the chain of spheres in the exceptional
  locus of the minimal resolution of a Wahl singularity.
\end{defi}

\begin{theo}
  Any T-string can be obtained from the string $[4]$ (corresponding
  to $p=2$, $q=1$) by a sequence of operations $L$ and $R$:
  \begin{align*}
    L[b_1,\ldots,b_{\ell}]&=[2,b_1,\ldots,b_{\ell-1},b_{\ell}+1],
    R[b_1,\ldots,b_{\ell}]&=[b_1+1,b_2,\ldots,b_{\ell},2].
  \end{align*}
  Let us define $F(x,y)=x^2/(xy-1)$. If $b=[b_1,\ldots,b_\ell]$ is the
  continued fraction expansion of $F(p,q)$ then $Lb$ is the continued
  fraction expansion of $F(p+q,q)$ and $Rb$ is the continued fraction
  expansion of $F(p+(p-q),p-q)$.
\end{theo}

\begin{coro}\label{coro:wahlcomb}
  If $[b_1,\ldots,b_\ell]$ is a T-string then
  \[\sum_{j=1}^\ell(b_j-2)=\ell+1.\]
\end{coro}
\begin{proof}
  This is true in the base case $[4]$ and is preserved by the
  operations $L$ and $R$, so is true for all T-strings.
\end{proof}

\section{Almost complex structures for irregular curves}\label{sct:acs}

A version of the following result was proved by McDuff and Opshtein
{\cite[Definitions 1.2.1, 2.2.1 and Lemma
    2.2.3]{McDuffOpshtein}}. Because we require slightly more of our
almost complex structures, we explain the proof here.

\begin{lemm}\label{lemm:cpxstr}
  There exists a nonempty set $\mathcal{J}_{reg}(\mC,\kappa)$ of
  $\omega$-tame almost complex structures $J$ on $\tilde{X}$ with the
  following properties:
  \begin{itemize}
  \item there is a neighbourhood $\nu$ of $\mC$ on which $J$
    is equal to the standard complex structure $J_0$ on the minimal
    resolution of the $\frac{1}{p^2}(pq-1,1)$ singularity. In
    particular, the symplectic spheres $C_1,\ldots,C_\ell$ are all
    $J$-holomorphic, and there is a $J$-holomorphic projection map
    $\rho\colon\nu\to\CC^2/\Gamma_{p,q}$ which contracts these spheres
    to the origin and is injective elsewhere.
  \item the image of any nonconstant genus zero $J$-holomorphic curve
    $D$ with $\int_D\omega<\kappa$ and $K\cdot D\geq 0$ is necessarily
    contained $C_1\cup\cdots\cup C_\ell$.  Indeed, if
    $b^+(\tilde{X})>1$, then the only embedded $J$-holomorphic spheres
    in $\tilde{X}$ with energy less than $\kappa$ are
    $C_1,\ldots,C_\ell$ and possibly a finite collection of pairwise
    disjoint embedded $-1$-spheres.
  \end{itemize}
\end{lemm}
\begin{proof}
  Fix the standard (integrable) complex structure $J_0$ on a
  neighbourhood $\nu$ of $C_1\cup\cdots\cup C_{\ell}$ and let
  $\rho\colon\nu\to\CC^2/\Gamma_{p,q}$ denote the holomorphic map
  which contracts the curves $C_i$ to the point $0$.

  The next lemma follows immediately from the proof of
  {\cite[Proposition 3.2.1]{McDuffSalamon}}:
  
  \begin{lemm}
    Let $\mathcal{J}(\nu)$ denote the space of almost complex
    structures on $\tilde{X}$ which agree with $J_0$ on $\nu$. There
    is a residual subset
    $\mathcal{J}_{reg}(\mC,\kappa)\subset\mathcal{J}(\nu)$ such
    that for any $J\in\mathcal{J}_{reg}(\mC,\kappa)$, any
    somewhere-injective irregular $J$-holomorphic sphere with energy
    less than $\kappa$ is contained in $\nu$.
  \end{lemm}

  In four dimensions, a sphere in the class $D$ is regular only if
  $K\cdot D\leq -1$ (otherwise the virtual dimension of its moduli space
  is negative). If $D$ is a $J$-holomorphic sphere with $K\cdot D\geq 0$
  then the underlying somewhere-injective curve $D'$ also has
  $K\cdot D'\geq 0$ and is therefore irregular. Therefore if
  $J\in\mathcal{J}_{reg}(\mC,\kappa)$ and $D$ is a $J$-sphere
  with $K\cdot D\geq 0$ and energy less than $\kappa$ then $D$ is
  contained in $\nu$.
  
  If $D$ is any nonconstant $J$-holomorphic curve in $\tilde{X}$ which
  is completely contained in $\nu$ then its image under $\rho$ is a
  holomorphic curve in $\CC^2/\Gamma_{p,q}$, and is therefore
  constant, so $D\subset C_1\cup\cdots\cup C_\ell$ as required.

  Suppose $b^+(\tilde{X})>1$. If $D$ is a $J$-sphere in $\tilde{X}$
  with energy less than $\kappa$ then, by Corollary \ref{coro:rat},
  either $D$ is an embedded $-1$-sphere (of which there is a finite
  collection and they are all pairwise disjoint, by Lemma
  \ref{lemm:distinctEclasses}) or else $K\cdot D\geq 0$ and by
  what we have proved so far, $D$ is one of the spheres
  $C_1,\ldots,C_{\ell}$.
\end{proof}

\section{Topological obstructions}\label{sct:topology}

\subsection{Discrepancies}

Consider the minimal resolution of the singularity
$\frac{1}{p^2}(pq-1,1)$, with exceptional locus
$C_1,\ldots,C_\ell$. Let $\tilde{U}$ be a neighbourhood of the
exceptional locus, and let $\Sigma$ be the boundary of
$\tilde{U}$. Alexander-Lefschetz duality tells us that
\[H^2(\tilde{U};\QQ)\cong H_2(\tilde{U},\Sigma;\QQ)\cong
H_2(\tilde{U};\QQ),\]
as $H_*(\Sigma;\QQ)$ is concentrated in degree zero. In particular, we
can write $K_{\tilde{U}}$ as a rational linear combination of the
classes $C_j$, Poincar\'{e}-dual to the corresponding curves:
\[K_{\tilde{U}}=\sum_{j=1}^\ell a_jC_j.\]
The coefficients $a_j$ are called the {\em discrepancies} of the
singularities. Note that discrepancies will be rational (non-integer)
numbers because, over $\ZZ$, the sublattice
$H_2(\tilde{U};\ZZ)\subset H_2(\tilde{U},\Sigma;\ZZ)$
has index $p^2$. We can calculate the discrepancies in terms of the
T-string $[b_1,\ldots,b_{\ell}]$ by solving the system of simultaneous
equations:
\begin{align}
  \label{eq:sim1}b_1-2=K_{\tilde{U}}\cdot C_1&=-a_1b_1+a_2\\
\nonumber  b_2-2=K_{\tilde{U}}\cdot C_2&=a_1-a_2b_2+a_3\\
\nonumber  \vdots&\\
\nonumber  b_{\ell-1}-2=K_{\tilde{U}}\cdot
\nonumber  C_{\ell-1}&=a_{\ell-2}-a_{\ell-1}b_{\ell-1}+a_{\ell}\\
\nonumber b_{\ell}=K_{\tilde{U}}\cdot C_{\ell}&=a_{\ell-1}-a_\ell b_\ell.
\end{align}
For example, when $\ell=1$ and the T-string is $[4]$, we get
$a_1=-\tfrac{1}{2}$.

Note that this definition necessarily agrees with the usual
algebro-geometric definition of discrepancies because, in both cases,
the $a_j$ are determined by the simultaneous equations \eqref{eq:sim1}.

Wahl singularities are log terminal, which means that $a_j\in(-1,0)$
for all $j$. The discrepancies of Wahl singularities are discussed
extensively in \cite{Kawamata,Lee}, and from a more symplectic perspective in \cite{McLean_Reeb}. The only property we will use is the following:

\begin{lemm}[{\cite[Corollary 3.2]{Kawamata}}]\label{lemm:kawamata}
 \[a_1+a_\ell=-1.\]
\end{lemm}

\subsection{Mayer-Vietoris}

Let $\iota\colon B_{p,q}\to X$ be a symplectic embedding, let
$\orb{X}$ be the orbifold obtained by collapsing $\iota(B_{p,q})$, and $\tilde{X}$ be the minimal resolution of $\orb{X}$. As in
the previous section, let $\tilde{U}$ be a neighbourhood of the
exceptional locus $\mC=C_1\cup\ldots\cup C_\ell$ of the
resolution; let $V\subset\tilde{X}$ be the complement of $\tilde{U}$,
and let $\Sigma$ be the interface between $\tilde{U}$ and $V$. As the
boundary of $\tilde{U}$ is a rational homology sphere, the
Mayer-Vietoris sequence for $\tilde{X}=\tilde{U}\cup V$ over $\QQ$
gives
\[H^2(\tilde{X};\QQ)=H^2(\tilde{U};\QQ)\oplus H^2(V;\QQ).\]
In terms of this decomposition, we have
\[K_{\tilde{X}}=\left(K_{\tilde{U}},K_V\right),\]
because the first Chern class is natural under pullbacks.

\begin{lemm}\label{lemm:multcap}
  If $F$ is a cycle in $\tilde{X}$, there is a (closed) cycle $F'$ in
  $V$, obtained by multiplying $F$ by $p^2$, slicing it along $\Sigma$
  and capping off the result, such that
  \[(0,K_V)\cdot F=\frac{1}{p^2}K_V\cdot F'.\]
\end{lemm}
\begin{proof}
  Consider the composition $\Phi$ of maps
  \[H_2(\tilde{X};\ZZ)\to H_2(\tilde{X},\tilde{U};\ZZ)\cong
  H_2(V,\Sigma;\ZZ)\to H_1(\Sigma;\ZZ).\]
  Since $H_1(\Sigma;\ZZ)=\ZZ/(p^2)$, we see that $\Phi(p^2F)=0$. Let
  $s$ be the 1-cycle in $\Sigma$ which is the image of $p^2F$ under
  the chain-level version of $\Phi$. Pick a 2-chain $P$ in $\Sigma$
  such that $\partial P=-s$. Now the chain $F'=p^2F+P$ is a (closed, not relative) cycle
  in $V$. We have
  \[K_V\cdot F=\frac{1}{p^2}K_V\cdot F'.\]
 To see that there is no
  contribution to the intersection pairing from the 2-chain $P$, note
  that, after multiplying both sides of the equality by $p^2$ to make
  the canonical bundle trivial and not just torsion in a neighbourhood
  of $\Sigma$, one could represent the canonical class $K_V$ by the
  first Chern form of a connection which was flat in a neighbourhood
  of $\Sigma$.
\end{proof}

Assuming that the 1-cycle $s$ considered in the proof of Lemma
\ref{lemm:multcap} is a combination of Reeb orbits in $\Sigma$ for the
standard contact form, we now explain how to choose the caps to have
positive symplectic area. This will be used in Section
\ref{sct:negmon} below to find restrictions on how certain holomorphic
curves can intersect $\mC$.

\begin{lemm}\label{lemm:caps}
  Let $\alpha$ be the standard contact form on $\Sigma$ and let
  $\gamma$ be a closed Reeb orbit for $\alpha$. The $p^2$-fold cover
  of $\gamma$ is the asymptote of a holomorphic
  disc in $\left(\CC^2\setminus\{0\}\right)/\Gamma_{p,q}$.
\end{lemm}
\begin{proof}
  The Reeb orbits in the standard 3-sphere are in bijection with the
  possible slopes $[a:b]$ of complex lines in $\CC^2$: an affine
  complex line $ax+by+c=0$ in $\CC^2$ is asymptotic to the Reeb orbit
  corresponding to $[a:b]$. The action $\Gamma_{p,q}$ of $\rt_{p^2}$
  on $\CC^2$ gives an action of $\rt_{p^2}$ on the $\cp{1}$ of Reeb
  orbits; for example, under this action, the points $[1:0]$ and
  $[0:1]$ each have stabiliser isomorphic to $\rt_{p^2}$.

  Reeb orbits $\gamma$ in the quotient $S^3/\Gamma_{p,q}$ are in
  correspondence with $\cp{1}/\Gamma_{p,q}$. Let $\gamma$ be the Reeb
  orbit corresponding to $[a:b]\in\cp{1}/\Gamma_{p,q}$. The complex
  line $ax+by+c$ with $c\neq 0$ gives a holomorphic plane in
  $\left(\CC^2\setminus\{0\}\right)/\Gamma_{p,q}$ which is asymptotic
  to the $\OP{Stab}([a:b])$-fold cover of $\gamma$. If we let
  $m=p^2/\OP{Stab}([a:b])$ and precompose the plane with the $m$-fold
  branched cover $\CC\to\CC$, $z\mapsto z^{m}$, then we get a
  holomorphic plane asymptotic to $p^2$ times the Reeb orbit $\gamma$,
  as required.
\end{proof}

\subsection{Negative monotonicity}\label{sct:negmon}

\begin{lemm}\label{lemm:monotonicity}
  Suppose that $(X,\omega)$ is a negatively monotone symplectic
  manifold, that is (after possibly rescaling the symplectic form)
  $K_X=[\omega]$. Let $\iota\colon B_{p,q}\to X$ be a symplectic
  embedding, let $\orb{X}$ be the orbifold obtained by collapsing
  $\iota(B_{p,q})$, and $\rho\colon\tilde{X}\to\orb{X}$ be
  the minimal resolution of $\orb{X}$. Let
  $J\in\mathcal{J}_{reg}(\mC,\kappa)$. If $F\subset\tilde{X}$
  is a $J$-holomorphic curve such that $\rho(F)$ is nonconstant then
  $K_V\cdot F>0$.
\end{lemm}
\begin{proof}
  Since $J$ is standard on a neighbourhood of $\mC$, there is
  an almost complex structure $\orb{J}$ on $\orb{X}$ obtained by
  contracting $\mC$ to the singular point $p$, and a
  holomorphic map $\rho\colon\tilde{X}\to\orb{X}$. If $R$ is a Riemann
  surface and $u\colon R\to\tilde{X}$ is a $J$-holomorphic curve in
  $\tilde{X}$ then $\rho\circ u$ is a $\orb{J}$-holomorphic curve in
  $\orb{X}$. Let $Z=(\rho\circ u)^{-1}(p)$. The curve
  $\rho\circ u|_{R\setminus Z}$
  is a punctured holomorphic curve in a noncompact symplectic manifold
  where the noncompact end is modelled on a punctured neighbourhood of
  $0\in\CC^2/\Gamma_{p,q}$. This means that the punctures of
  $\rho\circ u|_{R\setminus Z}$ are asymptotic to covers of Reeb
  orbits for the standard contact form on $\Sigma$. Let
  $\phi\colon R'\to R$ be an $N$-fold branched cover such that
  each point $z\in Z$ is a branch point with multiplicity $p^2$ (there
  may also be other branch points and $R'$ may have higher genus
  than $R$). Let $Z'=\phi^{-1}(Z)$. Now
  $\rho\circ u\circ\phi|_{R'\setminus Z'}$
  is a punctured holomorphic curve which is asymptotic to covers of
  Reeb orbits where the covering multiplicity is a multiple of
  $p^2$. Now, just topologically, glue
  $\rho\circ u\circ\phi|_{R'\setminus Z'}$ to the holomorphic
  planar caps
  constructed in Lemma \ref{lemm:caps}. We get a topological surface
  with positive area in $V=\orb{X}\setminus\{p\}$ which is homologous
  to $N[u]$.
\end{proof}

\begin{theo}\label{theo:magic}
  Suppose that $K_X=[\omega]$. If
  $J\in\mathcal{J}_{reg}(\mC,\kappa)$ and $F\subset\tilde{X}$
  is a $J$-holomorphic curve then
  \[\sum_{j=1}^\ell a_jF\cdot C_j<K_{\tilde{X}}\cdot F.\]
  If $K_{\tilde{X}}\cdot F=-1$ (for example, if $F$ is an embedded
  symplectic $-1$-sphere) then the following intersection patterns
  between $F$ and the $C_j$ curves cannot occur:
  \begin{enumerate}
  \item[(1)] $F\cdot C_j=0$ for all $j\neq j_0$ and $F\cdot C_{j_0}=1$.
  \item[(2)] $F\cdot C_1=F\cdot C_\ell=1$ and $F\cdot C_j=0$ for
    $j\not\in\{2,\ldots,\ell-1\}$.
  \end{enumerate}
\end{theo}
\begin{proof}
  Since $\rho(F)$ is a nonconstant holomorphic curve in $\orb{X}$,
  Lemma \ref{lemm:monotonicity} tell us that
  $K_{V}\cdot F>0$. Therefore,
  \[0<K_V\cdot F=K_{\tilde{X}}\cdot F-\sum_{j=1}^\ell
  a_jC_j\cdot F\]
  This tells us that
  \begin{equation}\label{eq:magicineq}
    \sum_{j=1}^\ell a_jF\cdot C_j<K_{\tilde{X}}\cdot F.
  \end{equation}
  Now suppose that $K_{\tilde{X}}\cdot F=-1$.
  \begin{enumerate}
  \item[(1)] If $F\cdot C_j=0$ for all $j\neq j_0$ and
    $F\cdot C_{j_0}=1$ then Equation \eqref{eq:magicineq} implies
    \[-1<a_{j_0}<K_{\tilde{X}}\cdot F=-1,\]
    (using Lemma \ref{lemm:kawamata}) which gives a contradiction.
  \item[(2)] If $F\cdot C_1=F\cdot C_\ell=1$ and $F\cdot C_j=0$ for
    $j\in\{2,\ldots,\ell-1\}$ then Equation \eqref{eq:magicineq}
    implies (using Lemma \ref{lemm:kawamata})
    \[-1=a_1+a_\ell<K_{\tilde{X}}\cdot F=-1,\]
    which, again, gives a contradiction.
  \end{enumerate}
\end{proof}

\section{Exceptional spheres and their limits}\label{sct:exclim}

Let $S_1,\ldots,S_k$ be a maximal set of pairwise disjoint embedded
symplectic $-1$-spheres in $\tilde{X}$ and let $E_i=[S_i]$. Since
$b^+(\tilde{X})>1$, Lemma \ref{lemm:distinctEclasses} tells us that
there is a unique such set up to Hamiltonian isotopy and it contains
one symplectic $-1$-sphere from every possible isotopy class. Let
$\mathcal{E}=\{E_1,\ldots,E_k\}$.

\begin{theo}\label{theo:taubes}
  We have $k\geq \ell-K_X^2$, where $K_X$ is the canonical class of
  $X$.
\end{theo}
\begin{proof}
  We have $K_{\tilde{X}}^2=K_X^2-\ell$ and $K_S^2=K_{\tilde{X}}^2+k$,
  so
  \[K_S^2=K_X^2+k-\ell.\]
  By a theorem of Taubes {\cite[Theorem A(3)]{TaubesSWG}} (in the case
  $b^+>1$) and Liu {\cite[Main Theorem A]{LiuGompf}} (whenever the
  minimal model is not irrational ruled), since $S$ is minimal, we
  have $K_S^2\geq 0$. This implies $k\geq \ell-K_X^2$.
\end{proof}

Pick $\kappa\in\RR$ bigger than $\max_{i=1}^k\int_{S_i}\omega$ and let
$J\in\mathcal{J}_{reg}(\mC,\kappa)$. By definition (see Lemma
\ref{lemm:cpxstr}), the only somewhere-injective $J$-holomorphic
spheres with energy less than $\kappa$ are $C_1,\ldots,C_\ell$ and
possibly a collection of embedded $-1$-spheres. Let
$\{J_t\}_{t=1}^{\infty}$, be a sequence of almost complex structures
such that $\lim_{t\to\infty}J_t=J$ and such that for all $t<\infty$,
the classes $E_1,\ldots,E_k$ have embedded $J_t$-holomorphic
representatives $E_i(J_t)$ (this is possible because the space of
almost complex structures for which the homology classes $E_i$ are
represented by embedded holomorphic curves is dense in the space of
all tame almost complex structures). In the limit $t\to\infty$,
Gromov's compactness theorem asserts that the curves $E_i(J_t)$
converge to $J$-holomorphic stable maps.

\begin{defi}
  We will abuse notation and write $E_i$ for the $J$-holomorphic
  stable map in the class $E_i$. We say that $E_i$ is an {\em unbroken
    curve} if its domain is irreducible; otherwise, we say that $E_i$
  is a {\em broken curve}. We will write $\mE_{broken}\subset\mE$ for
  the subset of broken curves and $\mE_{unbroken}\subset\mE$ for the
  subset of unbroken curves. By Proposition \ref{prop:excfirst}, all of
  these curves are exceptional curves of the first kind because they
  all satisfy $K_{\tilde{X}}\cdot E_i=-1$ and they all inhabit
  homology classes which can be represented by embedded symplectic
  spheres.
\end{defi}

\begin{rema}
  At this point in the proof, we may appeal once again to
  {\cite[Theorem 3]{Sikorav}} and assume that \(J\) is integrable in a
  neighbourhood of {\em all} the curves (\(\mathcal{C}\) and
  \(\mathcal{E}\)) under consideration. This is important because, in
  what follows, when we talk about blowing down a curve, we mean the
  usual complex analytic blow-down.
\end{rema}

\subsection{Unbroken curves}

\begin{lemm}\label{lemm:excsphbound}
  Suppose that $E\in\mE_{unbroken}$ is an unbroken curve. Then:
  \begin{enumerate}
  \item[(a)] $\sum_{j=1}^\ell E\cdot C_j\geq 2$.
  \item[(b)] $E\cdot C_j\leq b_j-1$ with equality if and only if
    $b_j=2$ and $E\cdot C_j=1$.
  \end{enumerate}
\end{lemm}
\begin{proof}
  Part (a) follows from Theorem \ref{theo:magic}(1). See also
  {\cite[Section 3, Step 3]{Khodorovskiy}}.

  To prove part (b), let $S'$ be the result of blowing down the sphere
  $E$, and let $C'_j$ denote the image of $C_j$ under this blow-down
  map. We have
  \[K_{S'}\cdot C'_j=K_{\tilde{X}}\cdot C_j-E\cdot C_j=b_j-2-E\cdot C_j,\]
  and Corollary \ref{coro:rat} tells us that $K_{S'}\cdot C'_j\geq -1$
  with equality if and only if $C'_j$ is an embedded $-1$-sphere. This
  implies that $E\cdot C_j\leq b_j-1$ with equality if and only if
  $C'_j$ is an embedded $-1$-sphere, which can happen only if
  $E\cdot C_j=1$ (or else $C'_j$ fails to be embedded), in which
  case $b_j=2$.
\end{proof}

\subsection{Broken curves}\label{sct:broken}

Since we have chosen $J\in\mathcal{J}_{reg}(\mC,\kappa)$, the
image of a broken curve comprises a finite set of $-1$-spheres
$e_1,\dots,e_m$ together with a subset of the curves $C_j$ (there are
no other simple $J$-holomorphic curves). If $C_j$ appears as an
irreducible component of $E$ we will call it an {\em internal sphere}
and write $C_{int}$ for the sum of all internal spheres; otherwise we
call it an {\em external sphere} and write $C_{ext}$ for the sum of
all external spheres.

\section{Towards a bound}

The following lemma is a modification of a lemma of Rana {\cite[Lemma
    2.8]{Rana}} to the symplectic context.

\begin{lemm}\label{lemm:rana}
  We have:
  \[\sum_{i=1}^k\sum_{j=1}^\ell E_i\cdot C_j\leq \ell+1.\]
\end{lemm}
\begin{proof}
  We have $K_{\tilde{X}}\cdot C_j=b_j-2$ and
  \[\sum_{j=1}^\ell(b_j-2)=K_{\tilde{X}}\cdot\sum_{j=1}^\ell
  C_j=\ell+1\]
  by Corollary \ref{coro:wahlcomb}. The image of each
  $C_j$ under $\pi$ is a rational curve in a
  minimal symplectic manifold with $b^+(X)>1$, so by Corollary
  \ref{coro:rat}, $K_S\cdot\pi(C_j)\geq 0$, giving
  \[\pi^*K_S\cdot\sum_{j=1}^\ell C_j\geq 0.\]
  We have
  \[K_{\tilde{X}}=\pi^*K_S+\sum_{i=1}^kE_i,\]
  so
  \[\ell+1=K_{\tilde{X}}\cdot\sum_{j=1}^\ell C_j=\pi^*K_S\cdot
  \sum_{j=1}^\ell C_j+\sum_{i=1}^k\sum_{j=1}^\ell E_i\cdot C_j\geq\sum_{i=1}^k\sum_{j=1}^\ell E_i\cdot C_j.\]
\end{proof}

\begin{lemm}
  Let $J\in\mathcal{J}_{reg}(\mC,\kappa)$ and let $E\in\mE$ be
  a $J$-holomorphic exceptional curve of the first kind. We have
  \begin{equation}\label{eq:excbound}
    1\leq E\cdot\sum_{j=1}^\ell C_j
  \end{equation}
  with equality if and only if the following conditions hold.
  \begin{itemize}
  \item The curve $E$ has precisely one component $e$ with $e^2=-1$
    which intersects only two spheres $C_{x'}$ and $C_{y'}$ from the
    chain $C_1,\ldots,C_\ell$, and intersects both once transversely.
  \item The other components of $E$ are one of the following:

    \vspace{0.5cm}
    
    \begin{enumerate}
    \item[(A)] $C_1,C_2,\ldots,C_x,C_y,C_{y+1},\ldots,C_\ell$ for some
      $1\leq x'\leq x<y-1<y'\leq \ell$:

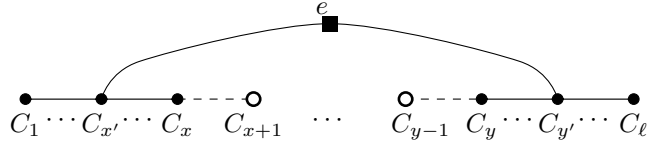
      \tikz[baseline=0]{
        \draw[fill=black,radius=2pt]
        (0,0) circle node [below=2pt] {$C_1$}
        --
        (0.5,0) node [below=2pt] {$\cdots$}
        --
        (1,0) circle node [below=2pt] {$C_{x'}$}
        --
        (1.5,0) node [below=2pt] {$\cdots$}
        --
        (2,0)
        circle node [below=2pt] {$C_x$};
        \draw[dashed]
        (2,0) node {}
        --
        (3,0) node [below=2pt] {$C_{x+1}$};
        \fill[fill=black,radius=3pt] (3,0) circle;
        \fill[fill=white,radius=2pt] (3,0) circle;
        \fill[fill=black,radius=3pt] (5,0) circle;
        \fill[fill=white,radius=2pt] (5,0) circle;
        \node at (4,0) [below=2pt] {$\cdots$};
        \draw[dashed]
        (5.2,0) node [below=2pt] {$C_{y-1}$}
        --
        (6,0) node {};
        \draw[fill=black,radius=2pt]
        (6,0) circle node [below=2pt] {$C_y$}
        --
        (6.5,0) node [below=2pt] {$\cdots$}
        --
        (7,0) circle node [below=2pt] {$C_{y'}$}
        --
        (7.5,0) node [below=2pt] {$\cdots$}
        --
        (8,0) circle node [below=2pt] {$C_\ell$};
        \draw plot [smooth] coordinates {(1,0) (1.5,0.5) (4,1) (6.5,0.5) (7,0)};
        \fill (3.9,0.9) node [above=3pt] {$e$} rectangle (4.1,1.1);
      }

      \vspace{1cm}
      
    \item[(B1)] $C_1,C_2,\ldots,C_x$ for some $1\leq x'\leq x<y'$:

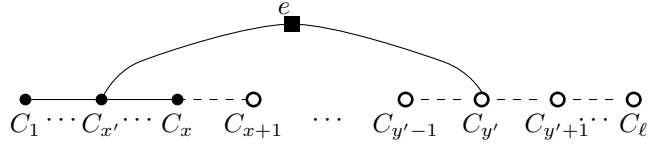
      \tikz[baseline=0]{
        \draw[fill=black,radius=2pt]
        (0,0) circle node [below=2pt] {$C_1$}
        --
        (0.5,0) node [below=2pt] {$\cdots$}
        --
        (1,0) circle node [below=2pt] {$C_{x'}$}
        --
        (1.5,0) node [below=2pt] {$\cdots$}
        --
        (2,0)
        circle node [below=2pt] {$C_x$};
        \draw[dashed]
        (2,0) node {}
        --
        (3,0) node [below=2pt] {$C_{x+1}$};
        \fill[fill=black,radius=3pt] (3,0) circle;
        \fill[fill=white,radius=2pt] (3,0) circle;
        \fill[fill=black,radius=3pt] (5,0) circle;
        \fill[fill=white,radius=2pt] (5,0) circle;
        \fill[fill=black,radius=3pt] (6,0) circle;
        \fill[fill=white,radius=2pt] (6,0) circle;
        \fill[fill=black,radius=3pt] (7,0) circle;
        \fill[fill=white,radius=2pt] (7,0) circle;
        \fill[fill=black,radius=3pt] (8,0) circle;
        \fill[fill=white,radius=2pt] (8,0) circle;
        \node at (4,0) [below=2pt] {$\cdots$};
        \node at (5,0) [below=2pt] {$C_{y'-1}$};
        \node at (6,0) [below=2pt] {$C_{y'}$};
        \node at (7,0) [below=2pt] {$C_{y'+1}$};
        \node at (7.5,0) [below=2pt] {$\cdots$};
        \node at (8,0) [below=2pt] {$C_{\ell}$};
        \draw[dashed]
        (5.2,0) node {}
        --
        (5.8,0) node {};
        \draw[dashed]
        (6.2,0) node {}
        --
        (6.8,0) node {};
        \draw[dashed]
        (7.2,0) node {}
        --
        (7.8,0) node {};
        \draw plot [smooth] coordinates {(1,0) (1.5,0.5) (3.5,1) (5.5,0.5) (6,0.1)};
        \fill (3.4,0.9) node [above=3pt] {$e$} rectangle (3.6,1.1);
      }

      \vspace{1cm}

    \item[(B2)] $C_y,C_{y+1},\ldots,C_\ell$ for some
      $1\leq x'<y\leq y'\leq \ell$:

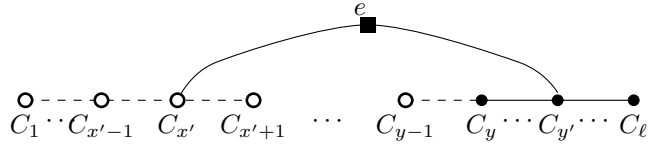
      \tikz[baseline=0]{
        \draw[fill=black,radius=2pt]
        (6,0) circle node [below=2pt] {$C_y$}
        --
        (6.5,0) node [below=2pt] {$\cdots$}
        --
        (7,0) circle node [below=2pt] {$C_{y'}$}
        --
        (7.5,0) node [below=2pt] {$\cdots$}
        --
        (8,0) circle node [below=2pt] {$C_{\ell}$};
        \draw plot [smooth] coordinates {(2,0) (2.5,0.5) (4.5,1) (6.5,0.5) (7,0.1)};
        \fill (4.4,0.9) node [above=3pt] {$e$} rectangle (4.6,1.1);
        \draw[dashed]
        (0,0) node [below=2pt] {$C_1$}
        --
        (1,0) node [below=2pt] {$C_{x'-1}$}
        --
        (2,0) node [below=2pt] {$C_{x'}$}
        --
        (3,0) node [below=2pt] {$C_{x'+1}$};
        \draw[dashed]
        (5,0) node [below=2pt] {$C_{y-1}$}
        --
        (6,0) node {};
        \fill[fill=black,radius=3pt] (0,0) circle;
        \fill[fill=white,radius=2pt] (0,0) circle;
        \fill[fill=black,radius=3pt] (1,0) circle;
        \fill[fill=white,radius=2pt] (1,0) circle;
        \fill[fill=black,radius=3pt] (2,0) circle;
        \fill[fill=white,radius=2pt] (2,0) circle;
        \fill[fill=black,radius=3pt] (3,0) circle;
        \fill[fill=white,radius=2pt] (3,0) circle;
        \fill[fill=black,radius=3pt] (5,0) circle;
        \fill[fill=white,radius=2pt] (5,0) circle;
        \node at (0.5,0) [below=2pt] {$\cdots$};
        \node at (4,0) [below=2pt] {$\cdots$};
      }

      \vspace{0.5cm}

    \end{enumerate}
  \end{itemize}
  In any of these equality cases, we say $E$ is a {\em bad curve of
    type (A), (B1) or (B2)}.
\end{lemm}
\begin{proof}
  For an unbroken curve, we know that
  $E\cdot\sum_{j=1}^\ell C_j\geq 2$ by Lemma \ref{lemm:excsphbound}(a), so
  we may assume that $E$ is broken.

  Recall from Section \ref{sct:broken} that $E$ comprises a finite set
  of $-1$-spheres $e_1,\ldots,e_m$ together with a collection of
  internal spheres (from amongst the $C_j$s). By Theorem
  \ref{theo:zariski}(6), we know that
  \begin{equation}\label{eq:bound1}
    -1\leq E\cdot C_{int}.
  \end{equation}
  ($C_{int}$ is the sum of the internal spheres; while that theorem
  also takes into account terms of the form $E\cdot e_i$, unless
  $E$ is unbroken, the spheres $e_i$ are all blown down before the
  final component, so $E\cdot e_i=0$).
  
  For a broken curve $E$, the intersection number $E\cdot C_{ext}$
  (where $C_{ext}$ is the sum of the external spheres) is
  greater than or equal to $\sum_{i=1}^me_i\cdot C_{ext}$ plus the
  number of interfaces between $C_{int}$ and $C_{ext}$ (it could be
  strictly greater if some of the components of $E$ come with higher
  multiplicity).

  Each sphere $e_i$ is itself an unbroken curve, so
  $\sum_{j=1}^\ell e_i\cdot C_j\geq 2$. In particular, there must be
  at least one external sphere, otherwise the graph $G$ defined in
  Theorem \ref{theo:zariski} would contain a cycle. This means that
  there is at least one interface between $C_{int}$ and $C_{ext}$, so
  $E\cdot\sum_{j=1}^\ell C_j\geq 0$.

  \begin{enumerate}
  \item[(A)] If $e_1$ does not intersect an external sphere, then it
    intersects two internal spheres. These spheres cannot be connected
    in the chain $C_1,\ldots,C_\ell$ by a sequence of spheres in
    $C_{int}$, or else the dual graph to $E$ would contain a cycle
    (contradicting Theorem \ref{theo:zariski}(3)), so in this case
    there would necessarily be two interfaces between $C_{int}$ and
    $C_{ext}$ and, again, we get $E\cdot\sum_{j=1}^\ell C_j\geq 1$.
    If equality holds then there are precisely two interfaces, and we
    deduce that $C_{int}=C_1+\cdots+C_x+C_y+\cdots+C_\ell$ for some
    $1\leq x'\leq x<x+1\leq <y-1<y\leq y'\leq\ell$, where $C_{x'}$ and
    $C_{y'}$ are the internal spheres hit by $e_1$.

  \item[(B)] If $e_1$ intersects an external sphere then we get
    $E\cdot\sum_{j=1}^\ell C_j\geq 1$. If we have equality then there
    is at most one interface between $C_{int}$ and $C_{ext}$, which
    means that either:
    \begin{enumerate}
    \item[(B1)] $C_{int}=C_1+\cdots+C_x$ for some
      $1\leq x<y'\leq\ell$, where $C_{y'}$ is the external sphere
      hit by $e_1$, or
    \item[(B2)] $C_{int}=C_y+\cdots+C_\ell$ for some
      $1\leq x'<y\leq\ell$, where $C_{x'}$ is the external sphere
      hit by $e_1$.
    \end{enumerate}
  \end{enumerate}
  If $m>1$ (so that there is a further $-1$-sphere $e_2$) then $e_2$
  either intersects an external sphere or else it connects two
  internal spheres, and to avoid creating a cycle in the dual graph we
  must find another interface between $C_{int}$ and $C_{ext}$. In
  either case, this pushes the inequality up to
  $E\cdot\sum_{j=1}^\ell C_j\geq 2$ as desired.
\end{proof}

\begin{lemm}\label{lemm:bound1}
  If there are precisely $p$ bad curves amongst the $E_i$ then
  \[\ell\leq 2K_X^2+p+1.\]
\end{lemm}
\begin{proof}
  This follows from the fact that the total number of blow-ups
  required is $k\geq \ell-K_X^2$, so at least $\ell-K_X^2-p$ of these
  blow-ups are associated to good curves. Therefore
  \[2(\ell-K_X^2-p)+p\leq \sum_{i=1}^k\sum_{j=1}^\ell E_i\cdot C_j,\]
  and the right hand side is less than or equal to $\ell+1$ by Lemma
  \ref{lemm:rana}. This gives
  \[\ell\leq 2K_X^2+p+1,\]
  as required.
\end{proof}

Therefore the problem of establishing bounds is reduced to the problem
of bounding the number of bad curves amongst the exceptional
divisors.

\section{Bounding bad curves}\label{sct:bounding}

The following is immediate from Proposition \ref{prop:nesting}.

\begin{coro}
  Any two bad curves of type (B1) share a common component. By
  Proposition \ref{prop:nesting}, we know that they are
  nested. Therefore, if there are any bad curves of type (B1), there
  is a maximal one with respect to nesting. The same holds for curves
  of type (B2), and it also follows that a bad curve of type (B1) and
  a bad curve of type (B2) cannot share a component.

  If there is a bad curve of type (A) then it shares components with
  any other bad curve, so in this situation there is a maximal bad
  curve with respect to nesting.
\end{coro}

\begin{coro}
  Suppose there are no bad curves of type (A). Let $n_1$ be the
  number of internal spheres in the maximal bad curve of type (B1)
  (zero if there are none) and $n_2$ the corresponding number for
  (B2). Then there are at most $n_1+n_2$ bad curves in total.

  If there is a bad curve of type (A), let $n$ be the number of
  internal spheres in the maximal one. Then there are at most $n$ bad
  curves in total.
\end{coro}
\begin{proof}
  Any two distinct bad curves $E_1$ and $E_2$ satisfy
  $E_1\cdot E_2=0$ by Lemma \ref{lemm:distinctEclasses}. Since any
  bad curve is contained in a neighbourhood $N$ of the maximal one (of
  its type, if there is no bad curve of type (A)), this means that if
  there are $p$ bad curves then there are $p+1$ homology classes in
  the homology of $N$ which are orthogonal with respect to the
  intersection product ($p$ coming from the bad curves, one coming
  from the $-1$-sphere which is necessarily there). Since the total
  rank of the homology of $N$ is equal to the number of internal
  spheres ($n$) plus one, we see that $p\leq n$.
\end{proof}

\begin{prop}\label{prop:badA}
  Let $E$ be a bad curve of type (A), maximal with respect to
  nesting. Suppose that it contains $n$ internal spheres. Then
  \[n\leq\frac{1}{2}(\ell+4).\]
\end{prop}
\begin{proof}
  We separate into a number of cases; any cases not explicitly listed
  here are related to one of the listed cases by symmetry
  (e.g. reflecting the diagram).
  \begin{enumerate}
  \item[(A.1)]
    
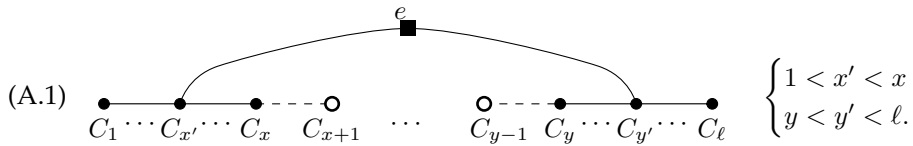
\tikz[baseline=0]{
      \draw[fill=black,radius=2pt]
      (0,0) circle node [below=2pt] {$C_1$}
      --
      (0.5,0) node [below=2pt] {$\cdots$}
      --
      (1,0) circle node [below=2pt] {$C_{x'}$}
      --
      (1.5,0) node [below=2pt] {$\cdots$}
      --
      (2,0)
      circle node [below=2pt] {$C_x$};
      \draw[dashed]
      (2,0) node {}
      --
      (3,0) node [below=2pt] {$C_{x+1}$};
      \fill[fill=black,radius=3pt] (3,0) circle;
      \fill[fill=white,radius=2pt] (3,0) circle;
      \fill[fill=black,radius=3pt] (5,0) circle;
      \fill[fill=white,radius=2pt] (5,0) circle;
      \node at (4,0) [below=2pt] {$\cdots$};
      \draw[dashed]
      (5.2,0) node [below=2pt] {$C_{y-1}$}
      --
      (6,0) node {};
      \draw[fill=black,radius=2pt]
      (6,0) circle node [below=2pt] {$C_y$}
      --
      (6.5,0) node [below=2pt] {$\cdots$}
      --
      (7,0) circle node [below=2pt] {$C_{y'}$}
      --
      (7.5,0) node [below=2pt] {$\cdots$}
      --
      (8,0) circle node [below=2pt] {$C_\ell$};
      \draw plot [smooth] coordinates {(1,0) (1.5,0.5) (4,1) (6.5,0.5) (7,0)};
      \fill (3.9,0.9) node [above=3pt] {$e$} rectangle (4.1,1.1);
    }$\quad\begin{cases}1<x'<x\\y<y'<\ell.\end{cases}$
    
    In this case, blowing down $e$ means that either $\pi_1(C_{x'})$
    or $\pi_1(C_{y'})$ becomes a $-1$-sphere which intersects three
    other components of the exceptional curve, which contradicts
    Theorem \ref{theo:zariski}(7).

  \item[(A.2)]
    
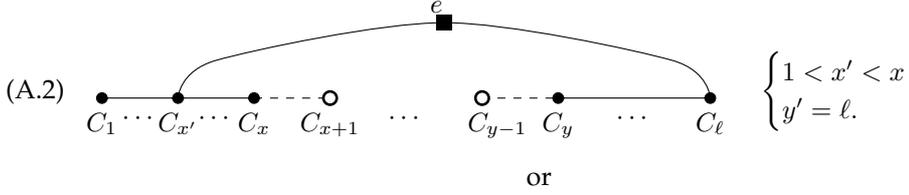
\tikz[baseline=0]{
      \draw[fill=black,radius=2pt]
      (0,0) circle node [below=2pt] {$C_1$}
      --
      (0.5,0) node [below=2pt] {$\cdots$}
      --
      (1,0) circle node [below=2pt] {$C_{x'}$}
      --
      (1.5,0) node [below=2pt] {$\cdots$}
      --
      (2,0)
      circle node [below=2pt] {$C_x$};
      \draw[dashed]
      (2,0) node {}
      --
      (3,0) node [below=2pt] {$C_{x+1}$};
      \fill[fill=black,radius=3pt] (3,0) circle;
      \fill[fill=white,radius=2pt] (3,0) circle;
      \fill[fill=black,radius=3pt] (5,0) circle;
      \fill[fill=white,radius=2pt] (5,0) circle;
      \node at (4,0) [below=2pt] {$\cdots$};
      \draw[dashed]
      (5.2,0) node [below=2pt] {$C_{y-1}$}
      --
      (6,0) node {};
      \draw[fill=black,radius=2pt]
      (6,0) circle node [below=2pt] {$C_y$}
      --
      (7,0) node [below=2pt] {$\cdots$}
      --
      (8,0) circle node [below=2pt] {$C_\ell$};
      \draw plot [smooth] coordinates {(1,0) (1.5,0.5) (4.5,1) (7.5,0.5) (8,0)};
      \fill (4.4,0.9) node [above=3pt] {$e$} rectangle (4.6,1.1);
    }$\quad\begin{cases}1<x'<x\\y'=\ell.\end{cases}$

    \begin{center}or\end{center}
      
      
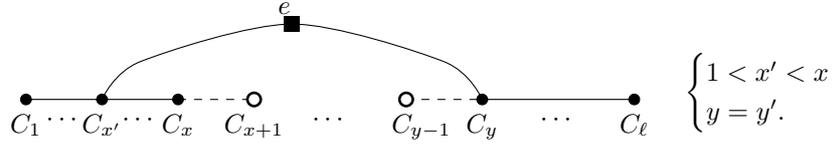
\tikz[baseline=0]{
        \draw[fill=black,radius=2pt]
        (0,0) circle node [below=2pt] {$C_1$}
        --
        (0.5,0) node [below=2pt] {$\cdots$}
        --
        (1,0) circle node [below=2pt] {$C_{x'}$}
        --
        (1.5,0) node [below=2pt] {$\cdots$}
        --
        (2,0)
        circle node [below=2pt] {$C_x$};
        \draw[dashed]
        (2,0) node {}
        --
        (3,0) node [below=2pt] {$C_{x+1}$};
        \fill[fill=black,radius=3pt] (3,0) circle;
        \fill[fill=white,radius=2pt] (3,0) circle;
        \fill[fill=black,radius=3pt] (5,0) circle;
        \fill[fill=white,radius=2pt] (5,0) circle;
        \node at (4,0) [below=2pt] {$\cdots$};
        \draw[dashed]
        (5.2,0) node [below=2pt] {$C_{y-1}$}
        --
  (6,0) node {};
        \draw[fill=black,radius=2pt]
  (6,0) circle node [below=2pt] {$C_y$}
        --
        (7,0) node [below=2pt] {$\cdots$}
        --
        (8,0) circle node [below=2pt] {$C_\ell$};
        \draw plot [smooth] coordinates {(1,0) (1.5,0.5) (3.5,1) (5.5,0.5) (6,0)};
        \fill (3.4,0.9) node [above=3pt] {$e$} rectangle (3.6,1.1);
      }$\quad\begin{cases}1<x'<x\\y=y'.\end{cases}$

      We handle these cases simultaneously. By the argument in Case
      (A.1), $C_{x'}$ cannot become a $-1$-sphere until all of the
      spheres $C_j$, $y\leq j\leq\ell$, have been blown down. In
      particular, this means that $C_y,\ldots,C_\ell$ is a chain of
      $-2$-spheres. Let us define $n_1:=\ell-y+1$ to be the number of
      the $-2$-spheres in this chain. We have
      $K\cdot\Pi_{n_1+1}(C_{x'})=K\cdot C_{x'}-n_1-1$ since we have
      blown down $e,C_y,\ldots,C_\ell$. Since $K\cdot
      C^{(n_1)}_{x'}=-1$, we see that
      \[K\cdot C_{x'}=n_1.\]
      Moreover,
      \[K\cdot C_1\geq n_1,\]
      because our T-string terminates in a chain of $n_1$
      $-2$-spheres. Once we have blown down $e,C_y,\ldots,C_\ell$, we
      are left with an exceptional curve
      $\Pi_{n_1+1}(C_1),\ldots,\Pi_{n_1+1}(C_x)$ with $x=n-n_1$ components,
      containing a single $-1$-sphere $\Pi_{n_1+1}(C_{x'})$. We take
      $S$ to be $\pi_{n_1+1}(C_{y-1})$, which has $K\cdot S\leq K\cdot
      C_{y-1}-1$ because $C_{y-1}$ has been attached to some curves
      which have been blown down. The configuration
      \[\Pi_{n_1+1}(C_1),\ldots,\Pi_{n_1+1}(C_x),S\]
      is precisely the configuration covered by Lemma
      \ref{lemm:iteratedblowdowns}. Therefore
      \[n-n_1\leq \frac{1}{2}(K\cdot C_{y-1}+2).\]
      Overall, we have
      \[n\leq
      \frac{1}{2}\left(K\cdot\left(C_1+C_{x'}+C_{y-1}\right)+2\right)\leq\frac{1}{2}(\ell+3),\]
      by Corollary \ref{coro:wahlcomb}.

    \item[(A.3)]
      
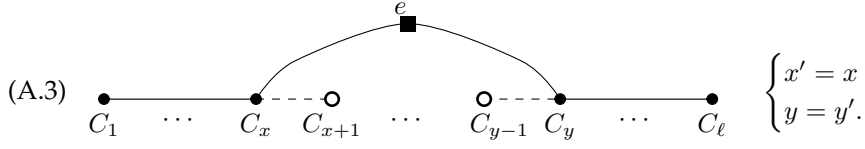
\tikz[baseline=0]{
        \draw[fill=black,radius=2pt]
        (0,0) circle node [below=2pt] {$C_1$}
        --
        (1,0) node [below=2pt] {$\cdots$}
        --
        (2,0)
        circle node [below=2pt] {$C_x$};
        \draw[dashed]
        (2,0) node {}
        --
        (3,0) node [below=2pt] {$C_{x+1}$};
        \fill[fill=black,radius=3pt] (3,0) circle;
        \fill[fill=white,radius=2pt] (3,0) circle;
        \fill[fill=black,radius=3pt] (5,0) circle;
        \fill[fill=white,radius=2pt] (5,0) circle;
        \node at (4,0) [below=2pt] {$\cdots$};
        \draw[dashed]
        (5.2,0) node [below=2pt] {$C_{y-1}$}
        --
        (6,0) node {};
        \draw[fill=black,radius=2pt]
        (6,0) circle node [below=2pt] {$C_y$}
        --
        (7,0) node [below=2pt] {$\cdots$}
        --
        (8,0) circle node [below=2pt] {$C_\ell$};
        \draw plot [smooth] coordinates {(2,0) (2.5,0.5) (4,1) (5.5,0.5) (6,0)};
        \fill (3.9,0.9) node [above=3pt] {$e$} rectangle (4.1,1.1);
      }$\quad\begin{cases}x'=x\\y=y'.\end{cases}$

      Let us, suppose without loss of generality, that $C_y$ is the
      second curve to be blown down in $E$. Then, if we set
      $S=\pi_1(C_{y-1})$, we find a configuration
      \[\pi_1(C_1),\ldots,\pi_1(C_x),\pi_1(C_y),\ldots\pi_1(C_\ell),S\]
      to which we can apply Lemma \ref{lemm:iteratedblowdowns} and get
      \[n\leq\frac{1}{2}(K\cdot C_{y-1}+3)\leq\frac{1}{2}(\ell+4),\]
      by Corollary \ref{coro:wahlcomb}.

    \item[(A.4)]
      
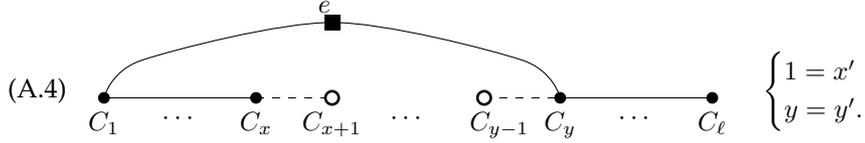
\tikz[baseline=0]{
        \draw[fill=black,radius=2pt]
        (0,0) circle node [below=2pt] {$C_1$}
        --
        (1,0) node [below=2pt] {$\cdots$}
        --
        (2,0)
        circle node [below=2pt] {$C_x$};
        \draw[dashed]
        (2,0) node {}
        --
        (3,0) node [below=2pt] {$C_{x+1}$};
        \fill[fill=black,radius=3pt] (3,0) circle;
        \fill[fill=white,radius=2pt] (3,0) circle;
        \fill[fill=black,radius=3pt] (5,0) circle;
        \fill[fill=white,radius=2pt] (5,0) circle;
        \node at (4,0) [below=2pt] {$\cdots$};
        \draw[dashed]
        (5.2,0) node [below=2pt] {$C_{y-1}$}
        --
        (6,0) node {};
        \draw[fill=black,radius=2pt]
        (6,0) circle node [below=2pt] {$C_y$}
        --
        (7,0) node [below=2pt] {$\cdots$}
        --
        (8,0) circle node [below=2pt] {$C_\ell$};
        \draw plot [smooth] coordinates {(0,0) (0.5,0.5) (3,1) (5.5,0.5) (6,0)};
        \fill (2.9,0.9) node [above=3pt] {$e$} rectangle (3.1,1.1);
      }$\quad\begin{cases}1=x'\\y=y'.\end{cases}$

      Let $n_1$ be the number of curves $C_1,\ldots,C_{n_1}$ which are
      blown down before the component $C_y$ is blown down. These
      necessarily form a (possibly empty) chain of $-2$-spheres. As in
      Case (A.2), we obtain
      \begin{align*}
        K\cdot C_y&=n_1\\
        K\cdot C_\ell&\geq n_1.
      \end{align*}
      As soon as $\Pi_{n_1+1}(C_y)$ becomes a $-1$-sphere, we can take
      $S=\Pi_{n_1+1}(C_{y-1})$ and apply Lemma
      \ref{lemm:iteratedblowdowns} to the configuration
      \[\Pi_{n_1+1}(C_x),\Pi_{n_1+1}(C_{x-1}),\ldots,\Pi_{n_1}1(C_{n_1+1}),\Pi_{n+1}(C_y),\ldots,\Pi_{n+1}(C_\ell),S\]
      to get
      \[n-n_1\leq\frac{1}{2}(K\cdot C_{y-1}+3).\]
      Overall,
      \[n\leq\frac{1}{2}\left(K\cdot\left(C_{y-1}+C_y+C_\ell\right)+3\right)\leq
      \frac{1}{2}(\ell+4).\]

    \item[(A.5)]
      
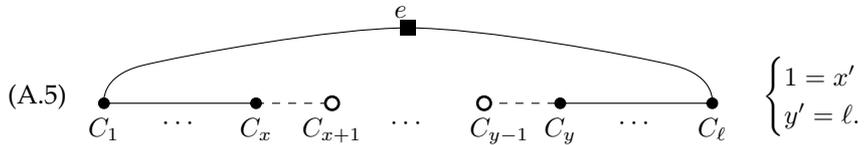
\tikz[baseline=0]{
        \draw[fill=black,radius=2pt]
        (0,0) circle node [below=2pt] {$C_1$}
        --
        (1,0) node [below=2pt] {$\cdots$}
        --
        (2,0)
        circle node [below=2pt] {$C_x$};
        \draw[dashed]
        (2,0) node {}
        --
        (3,0) node [below=2pt] {$C_{x+1}$};
        \fill[fill=black,radius=3pt] (3,0) circle;
        \fill[fill=white,radius=2pt] (3,0) circle;
        \fill[fill=black,radius=3pt] (5,0) circle;
        \fill[fill=white,radius=2pt] (5,0) circle;
        \node at (4,0) [below=2pt] {$\cdots$};
        \draw[dashed]
        (5.2,0) node [below=2pt] {$C_{y-1}$}
        --
        (6,0) node {};
        \draw[fill=black,radius=2pt]
        (6,0) circle node [below=2pt] {$C_y$}
        --
        (7,0) node [below=2pt] {$\cdots$}
        --
        (8,0) circle node [below=2pt] {$C_\ell$};
        \draw plot [smooth] coordinates {(0,0) (0.5,0.5) (4,1) (7.5,0.5) (8,0)};
        \fill (3.9,0.9) node [above=3pt] {$e$} rectangle (4.1,1.1);
      }$\quad\begin{cases}1=x'\\y'=\ell.\end{cases}$

      This configuration cannot occur by Theorem \ref{theo:magic}(2).
  \end{enumerate}
\end{proof}

\begin{prop}\label{prop:badB}
  Let $E$ be a bad curve of type (B1), maximal with respect to
  nesting. Suppose that $E$ contains $n$ internal spheres. Then
  \[n\leq\frac{1}{2}(\ell+4).\]
  The same inequality holds for bad curves of type (B2). If there are
  simultaneously bad curves $E_1$ of type (B1) and $E_2$ of type (B2)
  containing $n_1$, respectively $n_2$, internal spheres, then
  \[n_1+n_2\leq\frac{1}{2}(\ell+5).\]
\end{prop}
\begin{proof}
  When we consider only bad curves of type (B1), the argument is very
  similar to the arguments used in type (A): there are three cases:
  \begin{enumerate}
  \item[(B1.1)] $1=x'$,
  \item[(B1.2)] $1<x'<x$,
  \item[(B1.3)] $x'=x$.
  \end{enumerate}
  In cases (B1.1) and (B1.3), the internal spheres form a chain of
  $-2$-spheres of length $n_1$. In case (B1.1), this implies that
  \[n_1\leq K\cdot C_{y'}\mbox{ and }n_1\leq K\cdot C_{\ell}.\]
  Since $y'\neq\ell$ by Theorem \ref{theo:magic}(2), we find that
  \[n_1\leq\frac{1}{2}K\cdot (C_{y'}+C_{\ell})\leq \frac{1}{2}(\ell+1).\]
  In case (B1.3), assuming $x+1\neq \ell$, we get
  \[n_1-1\leq K\cdot C_{x+1}\mbox{ and }n_1\leq K\cdot C_\ell,\]
  so
  \[n_1\leq\frac{1}{2}(\ell+2).\]
  If $x+1=\ell$ then we necessarily have $y'=x+1$, so $e$ intersects
  $C_{x+1}$. When we blow down $e$, we find
  \[\pi_1(C_{\ell-1})\cdot\pi_1(C_{\ell})=2,\]
  so
  \[K\cdot\Pi_{n_1+1}(C_{x+1})=K\cdot C_{x+1}-1-2n_1\geq -1,\] which
  implies
  \[n_1\leq\frac{1}{2}K\cdot C_{x+1}\leq\frac{1}{2}(\ell+1).\]
  In case (B1.2), blowing down $e$ results in a configuration with
  $S=\pi_1(C_{y'})$ to which we can apply Lemma
  \ref{lemm:iteratedblowdowns} and deduce
  \[n_1\leq\frac{1}{2}(K\cdot C_{y'}+2)\leq\frac{1}{2}(\ell+3).\]

  It remains to understand what happens when we have a maximal (B1)
  curve $E_1$ and a maximal (B2) curve $E_2$. Let $e_1$ and $e_2$ be
  the $-1$-spheres in $E_1$ and $E_2$; we know that $E_1$ and $E_2$ do
  not share any components, so $e_1\neq e_2$ and there is no overlap
  between the internal spheres of $E_1$ and of $E_2$. Let $C_{x'_1}$,
  $C_{y'_1}$ be the spheres intersected by $e_1$ and $C_{x'_2}$,
  $C_{y'_2}$ be the spheres intersected by $e_2$. Let $C_1,\ldots,C_x$
  be the chain of internal spheres for $E_1$ and $C_y,\ldots,C_\ell$
  be the chain of internal spheres for $E_2$. We know that $e_1$ does
  not intersect any of $C_y,\ldots,C_\ell$ or else we would find a
  positive intersection between two exceptional classes, in
  contradiction to Lemma \ref{lemm:distinctEclasses}.

  Note that T-strings cannot both start and end with $2$, so if $E_1$
  is of type (B1.1) or (B1.3) then $E_2$ is of type (B2.2). Up to
  symmetry (switching the roles of $E_1$ and $E_2$) we can therefore
  assume that $E_2$ has type (B2.2).

  We now proceed according to the type of $E_1$:
  \begin{enumerate}
  \item[(B1.1)] In this case, we need to distinguish between the
    subcases $C_{y'_1}\neq C_{x'_2}$ and $C_{y'_1}=C_{x'_2}$. In the
    first case we get
    \[n_1\leq\frac{1}{2}K\cdot(C_{y'_1}+C_{\ell})\mbox{ and
      }n_2\leq\frac{1}{2}(K\cdot C_{x'_2}+3),\]
    so\[n_1+n_2\leq\frac{1}{2}\left(K\cdot\left(C_{x'_2}+C_{y'_1}+C_{\ell}\right)+3\right)\leq\frac{1}{2}(\ell+4).\]
    In the second case, we can blow down $E_1$ and get
    \[K\cdot\Pi_{n_1+1}(C_{y'_1})\leq K\cdot C_{y'_1}-n_1-1.\]
    Moreover, we know that $n_1\leq K\cdot C_{\ell}$.
    Subsequently blowing down $E_2$ yields
    \begin{align*}
      n_1+n_2&\leq n_1+\frac{1}{2}(K\cdot\Pi_{n_1+1}(C_{y'_1})+3)\\
      &\leq\frac{1}{2}(K\cdot C_{y'_1}+n_1-1+3)\\
      &\leq\frac{1}{2}(K\cdot (C_{y'_1}+C_{\ell})+2)\\
      &\leq\frac{1}{2}(\ell+3).
    \end{align*}
  \item[(B1.2)]  In this case, we need to distinguish between the
    subcases $C_{y'_1}\neq C_{x'_2}$ and $C_{y'_1}=C_{x'_2}$. In the
    first case we get
    \[n_1+n_2\leq\frac{1}{2}(K\cdot(C_{y'_1}+C_{x'_2})+4)\leq\frac{1}{2}(\ell+5).\]
    In the second case, we first blow down $E_1$ and look at the
    blow-down of $S=C_{y'_1}$. Arguing as in Lemma
    \ref{lemm:iteratedblowdowns}, we see that at the end of the blowing
    down process, the resulting rational curve $S'$ has
    \[K\cdot S'\leq K\cdot S-1-2(n_1-1).\]
    We now blow down $e_2$ and apply Lemma \ref{lemm:iteratedblowdowns}
    to the blow-down of $S'$. This gives
    \[n_2\leq\frac{1}{2}(K\cdot S-1-2(n_1-1)+3),\]
    or
    \[n_1+n_2\leq\frac{1}{2}(K\cdot C_{y'_1}+4)\leq\frac{1}{2}(\ell+5).\]
  \item[(B1.3)] In this case, we need to distinguish between the
    subcases $C_{x_1}\neq C_{x'_2}$ and $C_{x_1}=C_{x'_2}$. In the
    first case, we get
    \[n_1\leq K\cdot C_{x_1+1}+1,\ n_1\leq K\cdot
    C_{\ell},\ n_2\leq\frac{1}{2}(K\cdot C_{y'_1}+3),\]
    so
    \[n_1+n_2\leq\frac{1}{2}\left(K\cdot\left(C_{x_1+1}+C_\ell+C_{y'_1}\right)+4\right)\leq\frac{1}{2}(\ell+5).\]
    In the second case, we have again $n_1\leq K\cdot C_\ell$.
    Blow-down $E_1$, and look at the blow down $S$ of $C_{x_1+1}$
    along $E_1$. We have $K\cdot S\leq K\cdot C_{x_1+1}-n_1$. Let $S'$
    be the blow-down of $S$ along $e_2$; we have
    $K\cdot S'\leq K\cdot C_{x_1+1}-n_1-1$, and can apply Lemma
    \ref{lemm:iteratedblowdowns} to get
    \[n_2\leq\frac{1}{2}(K\cdot S'+3)\leq\frac{1}{2}(K\cdot
    C_{x_1+1}-n_1+2),\]
    so, overall,
    \[n_1+n_2\leq\frac{1}{2}\left(K\cdot\left(C_{x_1+1}+C_\ell\right)+2\right)\leq\frac{1}{2}(\ell+3).\]
  \end{enumerate}
\end{proof}

\begin{theo}
  We have $\ell\leq 4K_X^2+7$.
\end{theo}
\begin{proof}
  By Lemma \ref{lemm:bound1}, if there are $p$ bad curves, then
  \[\ell\leq 2K_X^2+p+1.\]
  We have seen in Propositions \ref{prop:badA} and \ref{prop:badB} that
  \[p\leq\frac{1}{2}(\ell+5),\]
  so
  \[\ell\leq 4K_X^2+7,\]
  as required.
\end{proof}

\section{Special case}\label{sct:special}

If we assume more about the form of the T-string then we get stronger bounds.

\begin{lemm}\label{lemm:khodcase}
  Let $X$ be a symplectic 4-manifold with $b^+(X)>1$ and suppose that
  $X$ contains a chain of $-2$-spheres $C_1,\ldots,C_n$ where $C_i$
  intersects $C_{i-1}$ and $C_{i+1}$ each once transversely and none
  of the other spheres in the chain. If $e$ is a $-1$-sphere in $X$
  then $e$ cannot intersect $C_2,\ldots,C_{n-1}$.
\end{lemm}
\begin{proof}
  Suppose $e$ intersects $C_i$ for $i\in\{2,\ldots,n-1\}$. If we blow
  down $e$ then $C_i$ becomes a rational curve $C'_i$ with
  $K\cdot C'_i=-e\cdot C_i$, so by Corollary \ref{coro:rat},
  $e\cdot C_i=1$ and
  $C'_i$ is an embedded $-1$-sphere. Blowing down $C'_i$ creates two
  new $-1$-spheres $C''_{i-1}$ and $C''_{i+1}$, and blowing down one
  of these turns the other into a sphere with self-intersection zero,
  in contradiction to Corollary \ref{coro:rat}.
\end{proof}

\begin{theo}
  For Wahl singularities whose T-string is $[2,\ldots,2,\ell+1]$, there are no bad curves at all, so
  \[\ell\leq 2K_X^2+1.\]
\end{theo}
\begin{proof}
  Suppose there is a bad curve $E$ containing a $-1$-sphere $e$. The
  sphere $e$ intersects two of the components in the chain. By Lemma
  \ref{lemm:khodcase}, $e$ can only intersect $C_1$, $C_{\ell-1}$ or
  $C_{\ell}$. It cannot intersect both $C_1$ and $C_\ell$ by Theorem
  \ref{theo:magic}(2) and it cannot intersect both $C_1$ and
  $C_{\ell-1}$ or else, upon blowing down, we create a sphere with
  self-intersection zero. Therefore the only possibility is that $e$
  intersects $C_{\ell-1}$ and $C_{\ell}$. Blowing down
  $e,C_{\ell-1},C_{\ell-2},\ldots,C_1$ in that order, the curve
  $C_{\ell}$ becomes a rational curve $C'_\ell$ with
  \[K\cdot C'_\ell=K\cdot C_\ell-2-2(\ell-1),\]
  (it decreases by one after blowing down $e$, and then intersects
  $C_{\ell-1}$ with multiplicity $2$, so it intersects all $\ell-1$ of
  the subsequent $-1$-spheres with multiplicity $2$). Since $K\cdot
  C_{\ell}=\ell-1$, this gives
  \[K\cdot C'_{\ell}\leq -1-\ell<-1,\]
  in contradiction to Corollary \ref{coro:rat}.
\end{proof}

\bibliography{khobib}
\bibliographystyle{plain}
\end{document}